\documentclass[11pt]{article}
\usepackage{amsmath,amssymb,color,amsthm}
\usepackage{authblk}
\usepackage{graphicx}
\usepackage{epsfig}
\usepackage{fullpage}
\usepackage{cite}
\usepackage{enumerate}
\usepackage{caption}
\usepackage{subcaption}
\usepackage{comment}

\numberwithin{equation}{section}
\usepackage{tikz}\usetikzlibrary{arrows}

\renewcommand{\geq}{\geqslant}

\newtheorem{theorem}{Theorem}[section]
\newtheorem{lemma}[theorem]{Lemma}
\newtheorem{proposition}[theorem]{Proposition}

\newtheorem{definition}[theorem]{Definition}

\newtheorem*{main-theorem}{Main Theorem}
\newtheorem*{remark*}{Remark}
\numberwithin{equation}{section}
\DeclareMathOperator{\tr}{Tr}

\begin{document}

\title{A Matrix Valued Kuramoto Model.}
\author[1]{Jared Bronski}
\author[2]{Thomas Carty}
\author[2]{Sarah Simpson}
\affil[1]{Department of Mathematics, University of Illinois}
\affil[2]{Department of Mathematics, Bradley University}

\date{\today}

\maketitle

\begin{abstract}
Beginning with the work of Lohe\cite{Lohe1,Lohe2} there have been a number of papers\cite{Choi.Ha.2014, deville2019synchronization,Ha.Ryoo.2016,Ha.Ko.Ryoo.2017,Ha.Ko.Ryoo.2018} that have generalized the Kuramoto model for phase-locking to a non-commuting 
situation. Here we propose and analyze another such model. We consider a collection of symmetric matrix-valued variables that evolve in such a way as to try to align their eigenvector frames.  The phase-locked state is one where the eigenframes all align, and thus the matrices all commute. We analyze the stability of the phase-locked state  and show that it is stable. We also analyze a dynamic analog of the twist states arising in the standard Kuramoto model, and show that these twist states are dynamically unstable.
\end{abstract}

\section{Introduction} \label{sec:intro}
The Kuramoto model,
\begin{equation}
\frac{d \theta_i}{dt}=\omega_i +\gamma \sum \limits_i \sin(\theta_j-\theta_i),
\label{eq:Kuramoto}
\end{equation}
where $\omega_i \in \mathbb{R}$ and $\gamma$ is the coupling constant, is a simplified model for many processes involving the sychronization of nonlinear coupled oscillators. This model has been used in a wide range of biological and physical applications, including modeling the phenomena of jet lag\cite{LKLAGO.2016} and circadian rhythms\cite{Childs.Strogatz.2008}, and describing electrical networks\cite{Dorfler.2005,Nardelli.2014}. Numerous variations of the Kuramoto model have been developed and studied, including a non-Abelian generalization of the Kuramoto model known as the Lohe model. Lohe actually proposed several related models:  the first was a flow on unitary matrices
\begin{equation}
i \dot{{\bf U}_j} {\bf U}_j^* = {\bf H}_j - \frac{iK}{2N} \sum \limits_{k=1}^{N} c_{kj}({\bf U}_j {\bf U}_k^* - {\bf U}_k {\bf U}_j^*),
\label{eq:LoheModel1}
\end{equation} where $j=1, ..., N$, $c_{kj} \in \mathbb{R}$, $k \geq 0$ is the coupling strength, ${\bf U}_j$ is a $d \times d$ unitary matrix, ${\bf U}_j^*$ is the Hermitian conjugate of ${\bf U}_j$, and ${\bf H}_j$ is a $d \times d$ Hermitian matrix  whose eigenvalues correspond to the natural frequencies of the Lohe oscillator at node $j$. This is Equation (2) in Lohe's 2009 paper \cite{Lohe1}. In this work Lohe demonstrated that in the one dimensional case, where $U_i=e^{-i\theta_i}$, an element of the Abelian unitary group $U(1)$, his model reduces to standard Kuramoto \eqref{eq:Kuramoto} model. In the same paper Lohe also proposes several other models based on other matrix groups. Additionally, other non-Abelian variations of the Kuramoto model have been developed which reduce to \eqref{eq:Kuramoto} or variants when an Abelian Lie group is chosen, including the quantum Kuramoto model\cite{deville2019synchronization}. This model is of the form
\begin{equation}
\frac{d}{dt} {\bf X}_i \dot {\bf X}_i^{-1} = {\bf \Omega}_i + \frac{1}{2} \sum \limits_{j=1}^n \gamma_{ij} \left( f ({\bf X}_j {\bf X}_i^{-1}) - f ({\bf X}_i {\bf X}_j^{-1})\right),
\label{eq:QuantumKuramoto}
\end{equation} where $f$ is a real analytic function, ${\bf X_i}$ is an element of some Lie group (satisfying a certain mild structural condition) and ${\bf \Omega}_i$ is an
element of the Lie algebra.

In a 2010 paper\cite{Lohe2} Lohe proposes a somewhat different model, Equation (10), in the form
\begin{equation}
  i \hbar \frac{\partial }{\partial t} |\psi_i \rangle = {\bf H}_i |\psi_i \rangle + \frac{i \hbar \kappa}{2N} \sum_j a_{ij} \left(|\psi_j \rangle\langle \psi_i|\psi_i\rangle  - |\psi_i\rangle \langle \psi_j | \psi_i \rangle \right).
\label{eqn:Lohe2}
\end{equation}
Here each wavefunction  $|\psi_i \rangle$ evolves according to a Hamiltonian ${\bf H}_i$, playing the role of the frequency in the standard Kuramoto model, and  the wavefunctions interact through a nonlinear coupling. Our model is most closely related to this model, though the evolution takes place at the level of the operators, not the wavefunctions, so it is worthwhile to spend a little time discussing the structure of these models. In particular these models take the general form of a conservative (Hamiltonian) part plus a constrained gradient descent.
\begin{equation}
  u_t = {\bf J} \delta_u H - \delta_u E - \sum_i \mu_i \delta_u C_i.
  \label{eq:AbstractLohe}
\end{equation}
Here $\delta_u$ represents the gradient in the finite dimensional case, and the Euler-Lagrange operator in the infinite dimensional case, ${\bf J}$ is some
skew-adjoint operator representing the Hamiltonian structure, $E$ is some functional to be minimized, attainment of the minimum of $E$ representing synchronization, $C_i$ representing some functionals to be conserved, and $\mu_i$ representing Lagrange  multipliers to enforce the constraints that $C_i(u)$ be constant.
The Hamiltonian flow is, obviously, the ${\bf H}_i |\psi_i\rangle $ term in Equation \eqref{eqn:Lohe2}. The remaining terms are a constrained gradient flow: these terms act to decrease the Dirichlet-like energy
  \[
    E_d = \frac{\kappa}{4N} \sum_{i,j} \langle\psi_i-\psi_j|\psi_i - \psi_j \rangle
    \]
    subject to the $N$ constraints that the wavefunctions $|\psi_i \rangle$ be properly normalized: $\langle \psi_i | \psi_i \rangle = \Vert \psi_i\Vert^2 = 1$. To see this we note that the gradient flow on the functional
    \[
      \tilde E_d = \frac{\kappa}{4N} \sum_{i,j} \langle \psi_i-\psi_j|\psi_i-\psi_j\rangle - \sum_{i=1}^N \mu_i \langle \psi_i | \psi_i \rangle.
      \]
      Here the first term, the Dirichlet energy, measures the degree of alignment of the wavefuntions $|\psi_i\rangle$, while the scalar quantities $\mu_i$ are  Lagrange multipliers to maintain the constraint $\langle \psi_i | \psi_i \rangle = 1$. Physically this can be interpreted as conserving the particle number of the $i^{th}$ species, and so $\mu_i$ is essentially the chemical potential associated with species  $i$.

      If we take the dynamics to be
      \[
        \frac{\partial }{\partial t} |\psi_i \rangle = J \frac{\partial}{\partial \langle\psi_i |} H - \frac{\partial}{\partial \langle \psi_i |} \tilde E_d
        \]
        with Hamiltonian $H = \sum_i \langle \psi_i| H_i | \psi_i\rangle$ and symplectic form $J = -\frac{i}{\hbar}$ then we find the dynamics
        \[
            i \hbar \frac{\partial }{\partial t} |\psi_i \rangle = {\bf H}_i |\psi_i \rangle + \frac{i \hbar \kappa}{2N} \sum_j a_{ij} |\psi_j-\psi_i\rangle - \mu_i |\psi_i\rangle.
          \]
          The Lagrange multipliers $\mu_i$ are determined by the condition that the particle numbers of individual species $\langle \psi_i | \psi_i \rangle$ be conserved: $\frac{\partial}{\partial t} \langle \psi_i | \psi_i \rangle=0$. This gives
\[
  i \hbar \frac{\partial }{\partial t} |\psi_i \rangle = {\bf H}_i |\psi_i \rangle + \frac{i \hbar \kappa}{2N} \sum_j a_{ij} \left(|\psi_j \rangle - |\psi_i\rangle \frac{\langle \psi_j | \psi_i \rangle}{\langle\psi_i|\psi_i\rangle } \right)
  \]which is equivalent to \eqref{eqn:Lohe2} under a trivial rescaling of $|\psi_i\rangle$.
Similarly to derive Equation (20) in Lohe's 2009 paper we take $u = \bigoplus_{i=1}^N {\bf x}_i $, $H = \frac 12\sum_i \Vert {\bf x}_i\Vert^2 = \Vert u\Vert^2 $, $ E = \frac{\kappa}{4N} \sum_{i,j} a_{ij}\Vert {\bf x}_i -{\bf x}_j\Vert^2$, ${\bf J} = \bigoplus_{i=1}^N {\bf \Omega}_i$ and the constraints $C_i$ to be $C_i({\bf x}_i) = \frac12 \Vert {\bf x}_i \Vert^2=\text{constant}$. Substituting these  into Equation (\ref{eq:AbstractLohe}) gives
\[
\frac{d{\bf x}_i}{dt} = \Omega_i {\bf x}_i +  \frac{\kappa}{N} \sum_{j} a_{ij}({\bf x}_j-{\bf x}_i) - \mu_i {\bf x}_i.
\]
In this spirit we consider a different generalization of the Kuramoto flow, where the phase space is $N$ $k \times k$ real symmetric matrices. We assume that these matrices evolve according to a gradient flow, where the energy functional is given by the sum of the Hilbert-Schmidt norms of the pairwise matrix commutators.  Since this is a gradient flow matrices evolve to minimize energy, whose global minimum is easily seen to be when the matrices pairwise commute.

Throughout this paper all matrices ${\bf M}_i$ are assumed to be real symmetric, while all matrices ${\bf \Omega}_i$ are assumed to be skew-symmetric.  In the remainder of Section \ref{sec:intro} we further motivate the current model.  We begin by deriving the matrix analog of the Kuramoto model with all frequencies equal.  We finish the introduction by establishing a concrete connection to a variant of the standard Kuramoto model.  In particular, in Section \ref{subsec:2X2case} we demonstrate that for the case of $2 \times 2$ matrices, this system can be described by a set of equations which is similar to the classical Kuramoto model with Hebbian interactions. In Section \ref{Twist States} we show that in the  $2 \times 2$ case, this model has a set of unstable fixed points associated with the case in which the matrices $\{{\bf M}_1, {\bf M}_2\ldots {\bf M}_n\}$ have eigenvectors pointing in directions which are evenly spaced about the unit circle. This is related to twist states in the classical Kuramoto model. Lastly, in Section \ref{Stability of Commuting Matrices} we show that for the set of pairwise commuting matrices of any size, this matrix-valued model has conditionally stable fixed points.

\subsection{Equations of motion and basic structure}
\label{subsec:BriefDerivation}

As argued earlier all of the models proposed by Lohe take the same abstract form: they consist of a conservative (Hamiltonian) part representing some abstract rotation, together with a constrained gradient flow part.  We follow this same philosophy to derive a non-commuting symmetric matrix valued flow.
We first define a gradient flow on a set of symmetric matrices. We will let $[\cdot,\cdot]$ denote the usual matrix commutator $[{\bf A},{\bf B}]={\bf A}{\bf B}-{\bf B}{\bf A}$ and recall that for real symmetric matrices the Hilbert-Schmidt norm is $\|{\bf A}\|^2=\tr {\bf A}^2$.  We now define an energy function $E$ to be a weighted sum of the Hilbert-Schmidt norms of all pair-wise commutators,
$$
E= \frac{1}{4} \sum \limits_{i,j} a_{ij} \tr\left([{\bf M}_{i},{\bf M}_{j}]^2\right).
$$
If we assume that $a_{ij}>0$ and symmetric then $E\geq 0$, with $E=0$ if and only if the matrices $\{{\bf M}_j\}_{j=1}^N$ all commute. The dynamics of the matrices ${\bf M}_i$ will be gradient flow on the function $E$.
\[
\frac{d{\bf M}_i}{dt} = -\frac{\partial E}{\partial {\bf M}_i}.
\]
The notation here is that if $E$ is a scalar valued function of a $k\times k$  matrix ${\bf M}$ then $\frac{\partial f}{\partial {\bf M}}$ is a $k\times k$ matrix whose entries are the derivatives of $E$ with respect to the corresponding entry of ${\bf M}$, $\left(\frac{\partial f}{\partial {\bf M}}\right)_{ij} =  \frac{\partial f}{\partial M_{ij}}.$ A straightforward calculation -- see Appendix \ref{app:DerivationOfFlow} for details -- gives the equations of motion as
\begin{equation}
  \frac{d{\bf M}_i}{dt} = - \sum_{j} a_{ij} [{\bf M}_j,[{\bf M}_i,{\bf M}_j]].
  \label{eq:MatrixMoto}
\end{equation}
These equations are not obviously related to the standard Kuramoto model, being cubic in the matrix entries instead of trigonometric, but we next show that, at least in the case of $2\times 2$ matrices the equations can be put into a form that is strongly reminiscent of the classical Kuramoto model. Specifically we show that in terms of the eigenvalues and the angles defining the eigenvectors the flow can be written as a Kuramoto flow for the angles together with a somewhat unusual Hebbian-type evolution for the eigenvalues, which play the role of a coupling strength.

More generally we can consider additional terms to this pure gradient flow representing a conservative/Hamiltonian dynamics and/or constraints. From the point of view of
matrices perhaps the most natural assumption is to constrain the Hilbert-Schmidt or
Euclidean norms of the operators: $\Vert {\bf M}_i\Vert^2 = \tr({\bf M}_i^2) = \text{\rm constant}$. This leads to a more general system of the form
\begin{equation}
  \frac{d{\bf M}_i}{dt} = [{\bf \Omega}_i, {\bf M}_i] - \sum_{j} a_{ij} [{\bf M}_j,[{\bf M}_i,{\bf M}_j]] - \mu_i {\bf M}_i.
  \label{eq:FullMatrixMoto}
\end{equation}
where the Lagrange multiplier $\mu_i,$ given by $\mu_i = \frac{\sum_j a_{ij} \tr({\bf M}_i [{\bf M}_j,[{\bf M}_i,{\bf M}_j]]))}{\tr({\bf M}_i^2)} = \frac{\sum_j 2 a_{ij} \tr(({\bf M}_i {\bf M}_j)^2 - ({\bf M}_i {\bf M}_j)^2)}{\tr {\bf M}_i^2},$ conserves the Hilbert-Schmidt norm $\tr({\bf M}_i^2)$ and ${\bf \Omega}_i$ is a real skew-symmetric matrix.  The  conservative term $[{\bf \Omega},{\bf M}_i]$
is the standard sort of term one sees in the evolution equations for an operator in the
interaction representation.

We note a few invariance properties of Equation (\ref{eq:FullMatrixMoto}) that generalize the analogous properties of the Kuramoto model. In the classical Kuramoto model the dynamics is invariant under the common rotation $\theta_i \mapsto \theta_i + \alpha$. Here this extends to orthogonal invariance: if one conjugates all ${\bf M}_i$ by the same orthogonal matrix ${\bf O}$, $\tilde {\bf M}_i={\bf O}{\bf M}_i{\bf O}^T$ then the new variables $\tilde {\bf M}_i$ satisfy (\ref{eq:FullMatrixMoto}) as well. Similarly in the classical Kuramoto if one moves to a uniformly rotating coordinate system $\theta_i  \mapsto \theta_i + \bar \omega t$ then this is the same as subtracting a mean from all of the frequencies: $\omega_i \mapsto \omega_i - \bar \omega$.  Here we have a similar invariance. If we move to a  coordinate system that rotates uniformly: $\tilde {\bf M}_i = e^{\tilde {\bf \Omega} t } {\bf M}_i e^{-\tilde {\bf \Omega} t }$, where $\tilde{\bf \Omega}$ is a skew-symmetric matrix, then (\ref{eq:FullMatrixMoto}) becomes
\begin{equation}
 \frac{d{\tilde {\bf M}}_i}{dt} = [{\bf \Omega}_i-\tilde{\bf \Omega}, {\tilde {\bf M}}_i] - \sum_{j} a_{ij} [{\tilde {\bf M}}_j,[{\tilde {\bf M}}_i,{\tilde {\bf M}}_j]] - \mu_i {\tilde {\bf M}}_i.
\label{eqn:MatrixMoto2}
\end{equation}

Note that the first invariance in particular has important implications for the spectrum of the linearization about a fixed point. The fact that the dynamics is invariant under conjugation by an element of the orthogonal group ${\bf O}(n)$ implies that the kernel of the linearized operator has dimension at least $\frac{n(n-1)}{2}$. We will discuss this further in a later section.

In most of this paper we will consider ${\bf \Omega}_i=0$, so we are considering strict synchronization rather than phase-locking. In our analysis of the twist states the constrained version of the equations arises, with the Lagrange multipliers arising from a change of variable rather than an explicit constraint.

\subsection{Spectral coordinates for the $2\times2$ case: reduction to Kuramoto form. }
\label{subsec:2X2case}

In order to build intuition we first consider the lowest dimensional case of possible interest, $2 \times 2$ real symmetric matrices, under the assumptions that $a_{ij}=1, \mu_i=0,{\bf \Omega}_i=0$ -- no constraints, no external frequency forcing and equally weighted all-to-all coupling.    In the case of the Lohe model when one restricts to the Abelian case the dynamics reduces to that of the classical Kuramoto. In our case the dynamics in the $2\times 2$ case is slightly more complicated, and can be expressed as a Kuramoto type evolution for the eigenangle, with the eigenvalues entering as an effective coupling strength between angles, together coupled with an evolution for the eigenvalues.

Let $\{{\bf M}_i\}$ be a collection of $N$ symmetric matrices. Using the fact that symmetric matrices are orthogonally diagnonalizable, we write each ${\bf M}_i$ as
$$ {\bf M}_i = {\bf R}(\theta_i){\bf \Lambda}_i {\bf R}(-\theta_i)= \begin{bmatrix}
\cos(\theta_i) & \sin(\theta_i)  \\
-\sin(\theta_i) & \cos(\theta_i)
\end{bmatrix}  \begin{bmatrix} \lambda_i^1 & 0 \\ 0 & \lambda_i^2 \end{bmatrix} \begin{bmatrix}
\cos(\theta_i) & -\sin(\theta_i)  \\
\sin(\theta_i) & \cos(\theta_i)
\end{bmatrix}.$$

Here $\lambda_i^1$ and $\lambda_i^2$ represent the two eigenvalues for the $i$-th matrix and $\theta_i$ represents the angle of rotation of the orthogonal eigenframe from the standard axes.

We now substitute these representations of ${\bf M}_i$ into \eqref{eq:MatrixMoto}.  It is somewhat more convenient  to work with the equivalent  form
$$
{\bf M}_i
=  \frac{1}{2}(\lambda_i^1 +\lambda_i^2) \begin{bmatrix}
1 & 0  \\
0 & 1 \end{bmatrix} + \frac{1}{2}(\lambda_i^2 -\lambda_i^1)\begin{bmatrix} -\cos(2\theta_i) & \sin(2\theta_i) \\ \sin(2\theta_i) & \cos(2\theta_i) \end{bmatrix}.
$$
For a fixed $i$, we take the derivative of ${\bf M}_i$ and get
$$
 \frac{d{\bf M}_i}{dt}=\frac{1}{2}(\dot{\lambda_i^1} +\dot{\lambda_i^2}) \begin{bmatrix}
1 & 0  \\
0 & 1 \end{bmatrix} + \frac{1}{2} (\dot{\lambda_i^2}-\dot{\lambda_i^1}) \begin{bmatrix} -\cos(2\theta_i) & \sin(2\theta_i) \\ \sin(2\theta_i) & \cos(2\theta_i) \end{bmatrix} + (\lambda_i^2-\lambda_i^1) \dot{\theta_i} \begin{bmatrix} \sin(2\theta_i) & \cos(2\theta_i) \\ \cos(2\theta_i) & -\sin(2\theta_i) \end{bmatrix}. \label{eq:2.3}
$$
After some algebra we find the following expression for the nested commutator $[{\bf M_j},[{\bf M}_i,{\bf M}_j]]$:
\begin{equation}
[{\bf M}_j,[{\bf M}_i,{\bf M}_j]]=\frac{1}{2} (\lambda_j^1-\lambda_j^2)^2 (\lambda_i^1-\lambda_i^2)\sin2(\theta_i-\theta_j) \begin{bmatrix} \sin(2\theta_j) & \cos(2\theta_j) \\ \cos(2\theta_j) & -\sin(2\theta_j) \end{bmatrix}.
\label{eq:2x2DoubleCommutator}
\end{equation}
Summing \eqref{eq:2x2DoubleCommutator} over all $j$, yields the representation of the right-hand side of \eqref{eq:MatrixMoto} under this decomposition.  The advantage of this representation is that we have passed the matrix dynamics purely onto the time-dependent eigenvalues and eigenangles.  In order to more clearly see this, we again take advantage of the structure of the symmetric matrices.
Note that the matrices ${\bf E}_1=\begin{bmatrix}
1 & 0  \\
0 & 1 \end{bmatrix}$,
${\bf E}_2=\begin{bmatrix} -\cos(2\theta_i) & \sin(2\theta_i) \\ \sin(2\theta_i) & \cos(2\theta_i) \end{bmatrix}$, and  ${\bf E}_3= \begin{bmatrix} \sin(2\theta_i) & \cos(2\theta_i) \\ \cos(2\theta_i) & -\sin(2\theta_i) \end{bmatrix}$ form an orthogonal basis for all $2\times 2$ symmetric matrices under the inner product $\langle A, B \rangle = \tr(A^T B)$. Using this basis, it is straightforward to find that the $N$ equations of \eqref{eq:MatrixMoto} are equivalent to the following evolution equations on the $2N$ eigenvalues $\lambda_i^1$, $\lambda_i^2$ and $N$ eigenangles $\theta_i:$
\begin{equation} \dot{\lambda_i^1}+\dot{\lambda_i^2}=0 \label{eq:2x2EigenvalueSumDynamics} \end{equation}
\begin{equation} \dot{\lambda_i^2}-\dot{\lambda_i^1}= - (\lambda_i^2-\lambda_i^1)\sum \limits_{j=1}^{N}(\lambda_j^2-\lambda_j^1)^2 \sin^2(2(\theta_j-\theta_i)) \label{eq:2x2CouplingDynamics} \end{equation}
\begin{equation} \dot{\theta_i} =
\sum \limits_{j=1}^{N}\frac{1}{4}(\lambda_j^2-\lambda_j^1)^2\sin(4(\theta_j-\theta_i)) \label{eq:2x2Eigen-angleDynamics} \end{equation}

More generally if one carries out the same calculation for the model with external frequency forcing and constraints, Equation \eqref{eq:FullMatrixMoto}, one finds
\begin{equation} \dot{\lambda_i^1}+\dot{\lambda_i^2}=-\mu_i (\lambda_i^1 + \lambda_i^2) \label{eq:2x2EigenvalueSumDynamics2} \end{equation}
\begin{equation} \dot{\lambda_i^2}-\dot{\lambda_i^1}= - (\lambda_i^2-\lambda_i^1)\sum \limits_{j=1}^{N}(\lambda_j^2-\lambda_j^1)^2 \sin^2(2(\theta_j-\theta_i)) - \frac12 \mu_i (\lambda_i^2-\lambda_i^1) \label{eq:2x2CouplingDynamics2} \end{equation}
\begin{equation} \dot{\theta_i} = \omega_i +
  \sum \limits_{j=1}^{N}\frac{1}{4}(\lambda_j^2-\lambda_j^1)^2\sin(4(\theta_j-\theta_i)) \label{eq:2x2Eigen-angleDynamics2}
\end{equation}
where $\Omega_i = \left[\begin{array}{cc}0 & \omega_i \\ -\omega_i & 0 \end{array}\right]$ and $\mu_i = \sum_j \frac{\tr({\bf M}_i [{\bf M}_j,[{\bf M}_i , {\bf M}_j]])}{\tr({\bf M}_i^2)}=- \frac12 \sum_j  \frac{(\lambda_j^1-\lambda_j^2)^2 (\lambda_i^1-\lambda_i^2)^2\sin^2 2 (\theta_i - \theta_j)}{(\lambda_i^1)^2 + (\lambda_i^2)^2}$. Note that $\mu_i$ is always negative so it tends to increase the size of the eigenvalues.

Together, Equations \eqref{eq:2x2CouplingDynamics} and \eqref{eq:2x2Eigen-angleDynamics} takes a form similar to that of classical Kuramoto, but with a dynamic coupling parameter that evolves in response to the phase-differences.  Such models are very natural in many applications, including neuroscience, and have received considerable attention in the literature \cite{bronski2017stability,ha2016synchronization,Isakov2014}.  In the theoretical neuroscience literature, such systems are called Hebbian if the coupling strength increases if the oscillators are in phase, and  decreases if the oscillators are out of phase, a behavior suggested in the classical neuroscience work of Hebb\cite{Hebb2005}. The behavior here is similar: in the absence of the constraint the coupling strength $(\lambda_j^2-\lambda_j^1)^2$ is always decreasing, but the rate of decrease goes to zero as the alignment difference $\theta_i-\theta_j$ approaches zero. Enforcing the constraint (in other words including the $\mu_i$ terms) gives a more classically Hebbian behavior -- the coupling strength is increasing in time for oscillators that are in phase and decreasing in time for oscillators that are out of phase.

That said, one obvious difference between (\ref {eq:2x2Eigen-angleDynamics}) and the classical Kuramoto model is a rescaling of the angle difference by 4: the righthand side of  (\ref {eq:2x2Eigen-angleDynamics}) is a periodic function of $\theta_i$ with period $\pi/2$ rather than period $2\pi, $ as in Kuramoto.  However, there is a clear geometric interpretation of this factor of four, since we are considering alignment of an orthogonal  frame rather than individual vectors. From the point of view of commutation it is not important which eigenvectors of ${\bf M}_i$ align with which eigenvectors of ${\bf M}_j$, so one expects a invariance under rotations through angle $\pi/2$, whence the factor of four. Another interesting observation is that, since the coupling between frames is a function of the eigenvalue difference $(\lambda_j^2-\lambda_j^1)$ there are two ways in which the system can move toward a fixed point: either the frames can align $\theta_i = \theta_j$ or the eigenvalues can
degenerate $\lambda_j^2=\lambda_j^1$ in which case ${\bf M}_j$ decouples from the
rest of the matrices, as it is a multiple of the identity and thus commutes with all the matrices. One fixed point of this system is the case in which all the angles, $\theta_i$, are equal. This leads to a set of matrices which all commute for any $\Delta \lambda_i$. Hence, the right hand side of \eqref{eq:MatrixMoto} is equal to zero and we are at a fixed point. We analyze the stability of these fixed points in Section \ref{Stability of Commuting Matrices}.

Figure $\ref{fig:commutatornorm2x2}$ demonstrates the dynamics of $\eqref{eq:MatrixMoto}$ through a numerical example for the case of three $2 \times 2$ matrices. In figure $\ref{fig:commutatornorm2x2}$ we show a plot of the Hilbert-Schmidt norm of the pairwise commutators of each of the matrices as a function of time together with a plot of the angles between the eigenvector frames as a function of time. It is clear that the matrices evolve to their minimum energy, which is the case in which all of the matrices commute pairwise.

The next figure, (\ref{fig:BigEvolution2x2}), depicts the evolution of the eigenframes for eight $2\times 2$ matrices. Each matrix is represented by a pair of lines. The lines point in the direction of the eigenvector, and the length is proportional to the corresponding eigenvalue, so it is possible to determine the eigenvectors and the magnitude (though not the sign) of the eigenvalues from such a plot. The eigenframes are depicted at times $t=0,.001,.002,.003,.004,.005$. Note that the evolution occurs relatively quickly in this case, as we have not scaled the interaction with $\frac{1}{N}$ as is customary with the Kuramoto model. One can see that the lengths
of the eigenvalues changes as the matrices evolve, and that there is a strong tendency for the eigenvalues of a given matrix to get closer together. (This is apparent from the equations, as the trace of each matrix is conserved, while the eigenvalue difference is a decreasing function.) We also note that the matrices with the largest eigenvalues tend to dominate the dynamics -- they ``pull'' the small matrices towards them. This is evident from the equations of motion, where the coupling strength between matrices $i$ and $j$ is proportional to the square of the difference of
the eigenvalues of ${\bf M}_j$.

\begin{figure}
\includegraphics[width = 0.45 \textwidth]{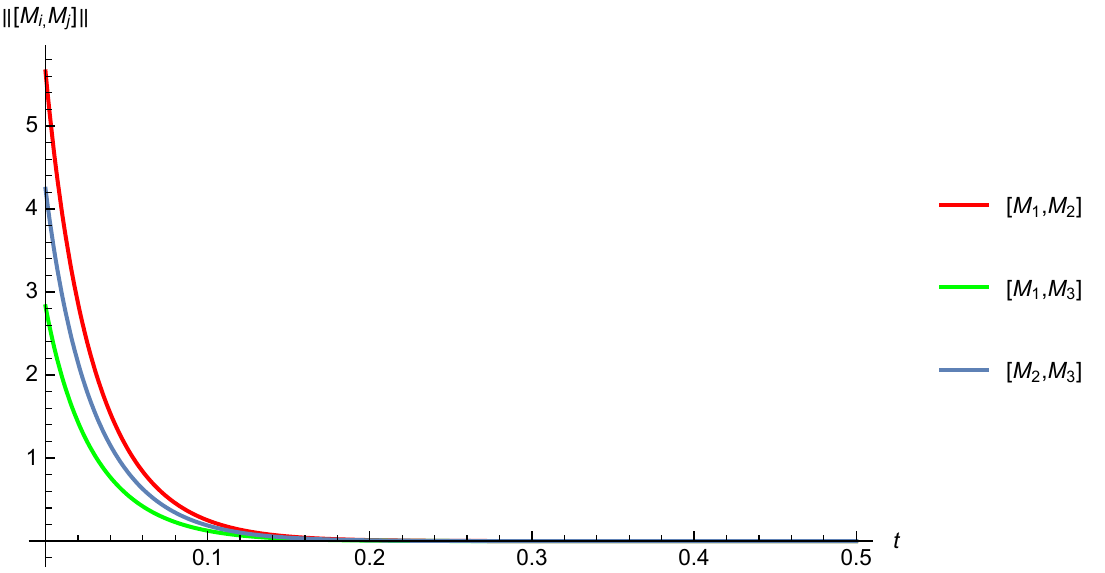}
\includegraphics[width = 0.45 \textwidth]{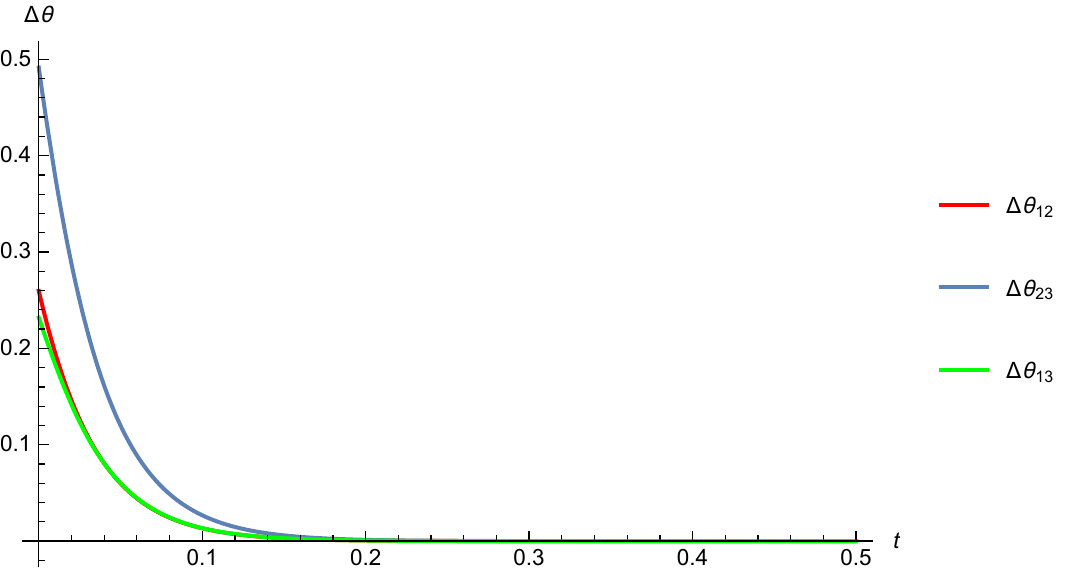}
\caption[]{The Hilbert-Schmidt norm of the pairwise commutators [left] and the angle between eigenframes [right] as a function of time in an example with three $2\times 2$ matrices. The initial conditions in this example are: $M_1(0)=$ $\begin{bmatrix} 1 & 2 \\ 2 & 3 \end{bmatrix}$, $M_2(0)= \begin{bmatrix} 2 & 1 \\ 1 & 5 \end{bmatrix}$, and $M_3(0)= \begin{bmatrix} 0 & 1 \\ 1 & 0 \end{bmatrix}$.}
\label{fig:commutatornorm2x2}
\end{figure}

\begin{figure}[htp]
\centering
\includegraphics[width=0.45\textwidth]{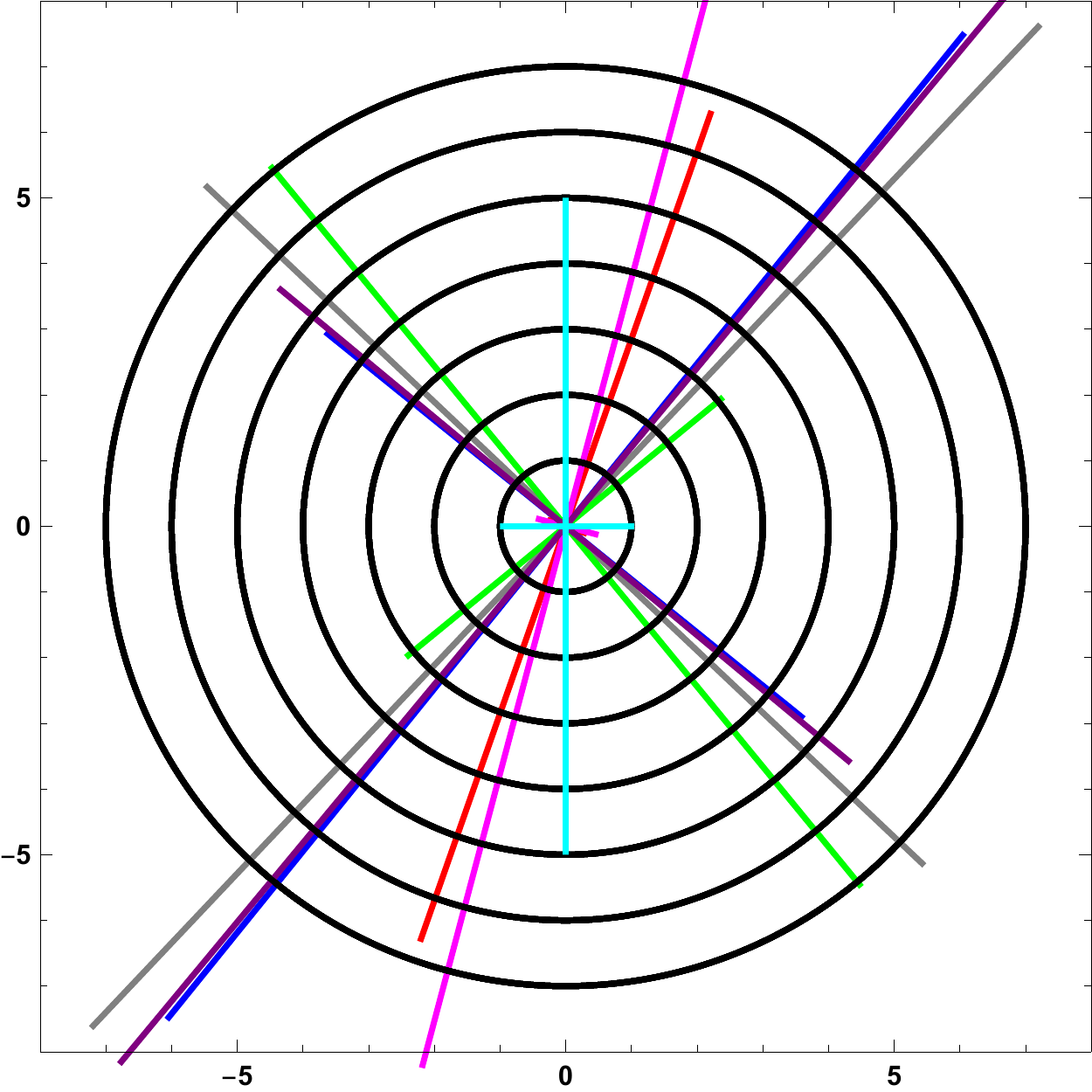}
\includegraphics[width=0.45\textwidth]{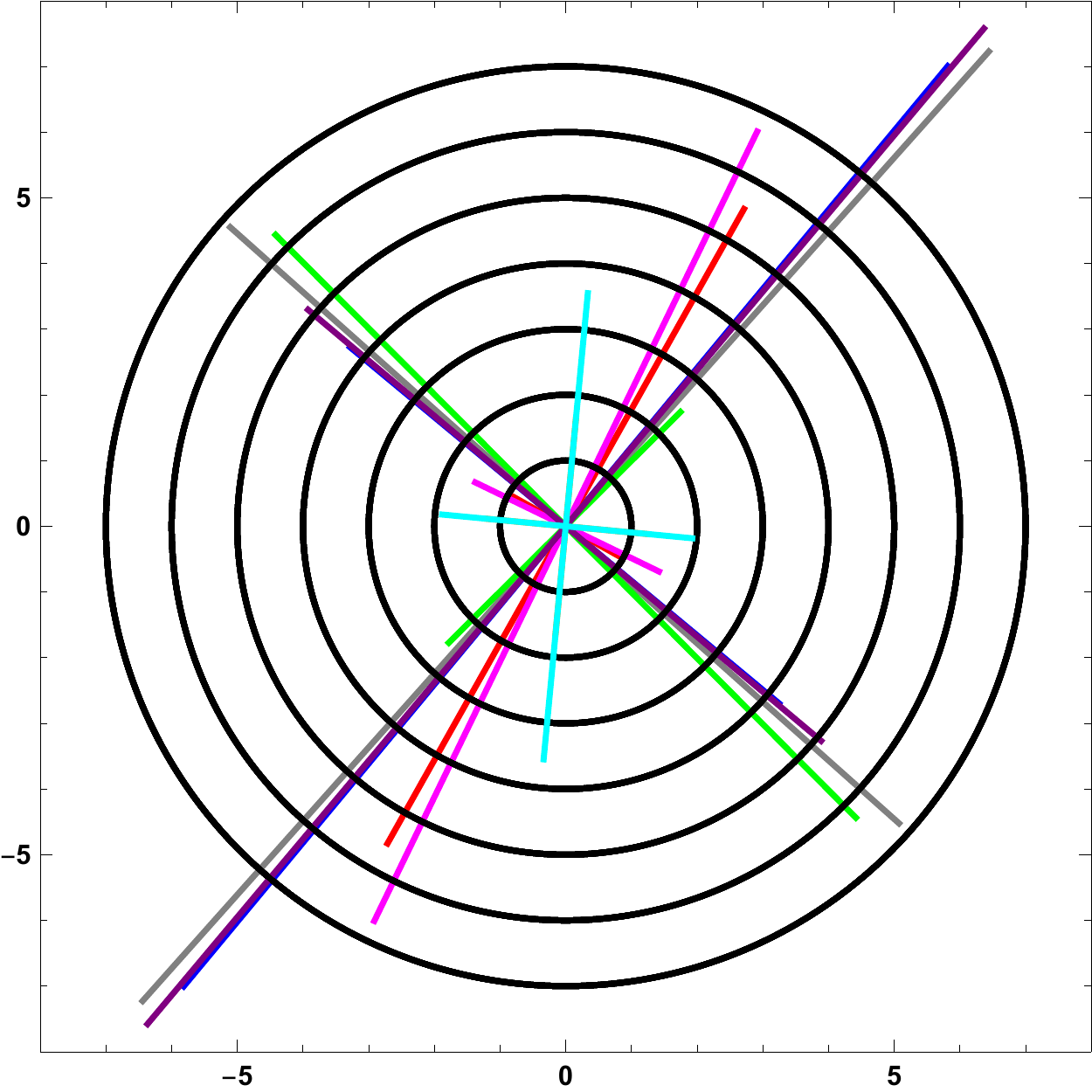} \\
\includegraphics[width=0.45\textwidth]{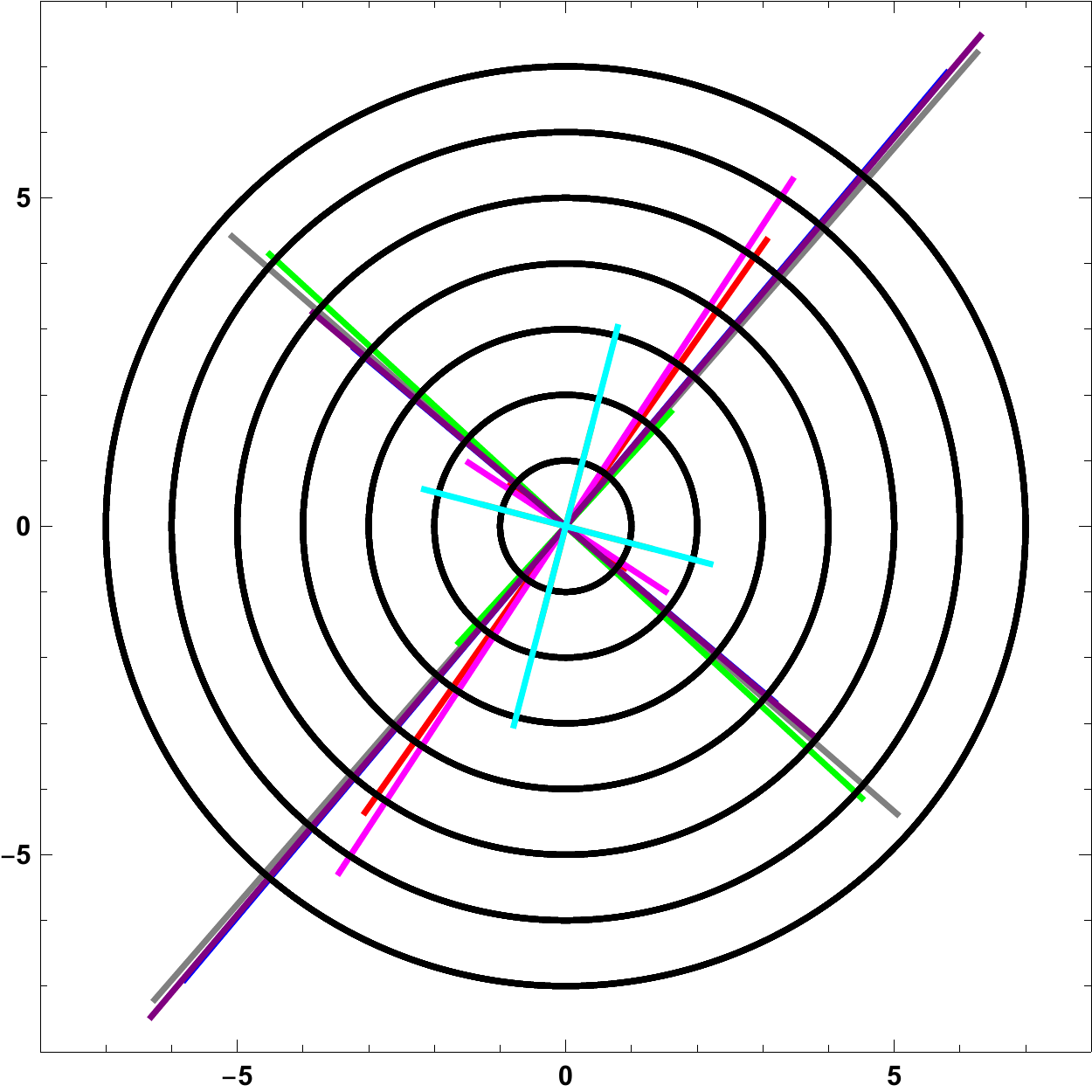}
\includegraphics[width=0.45\textwidth]{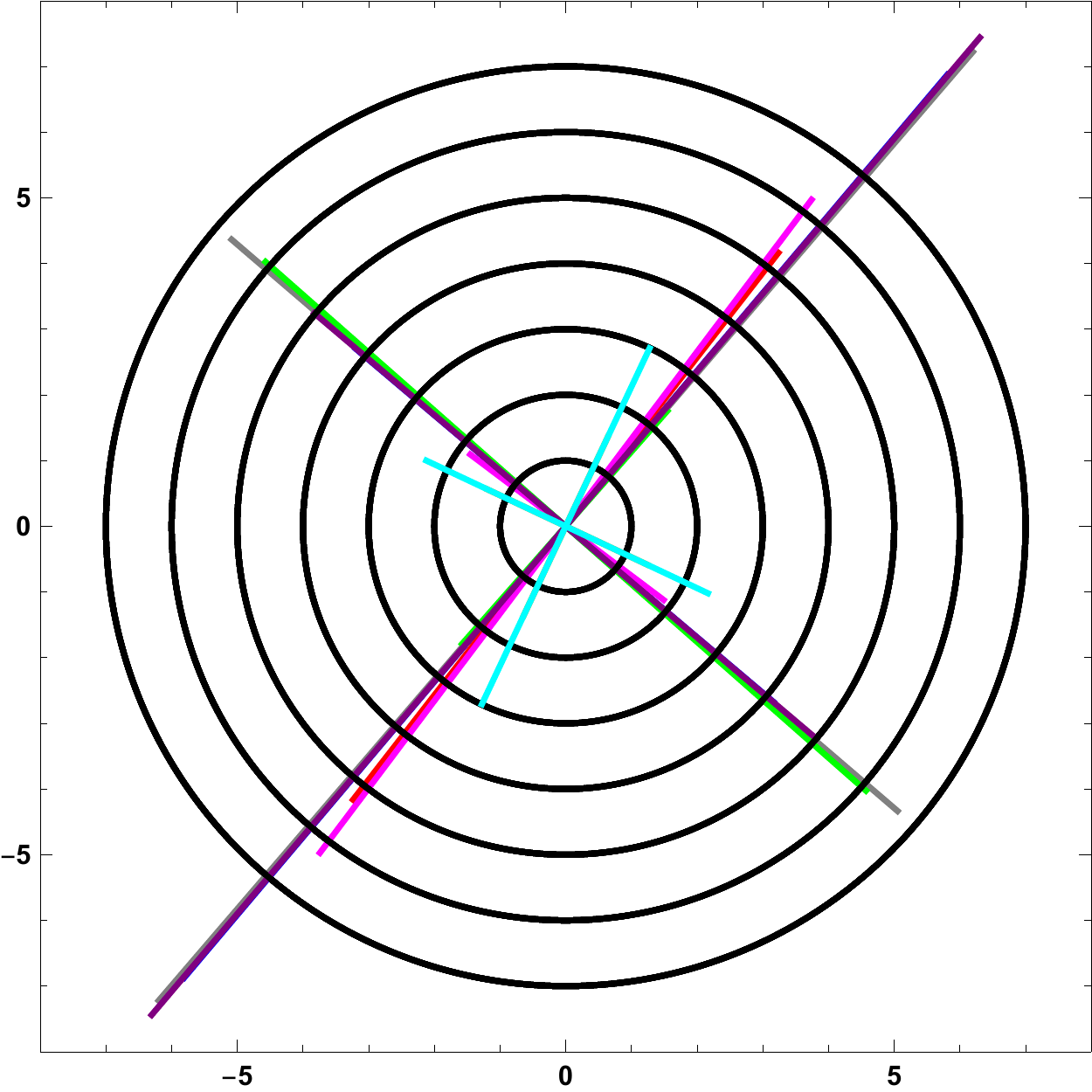} \\
\includegraphics[width=0.45\textwidth]{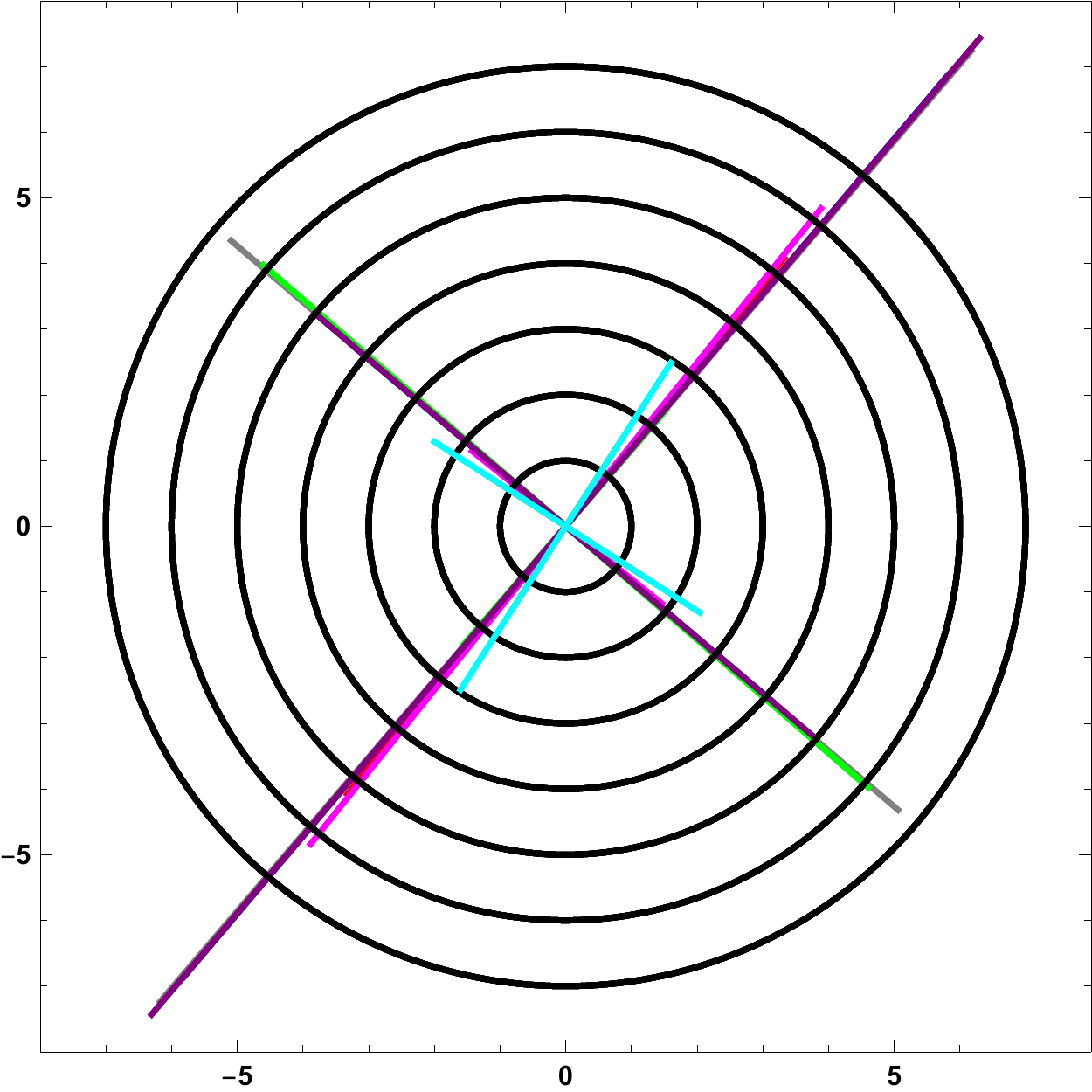}
\includegraphics[width=0.45\textwidth]{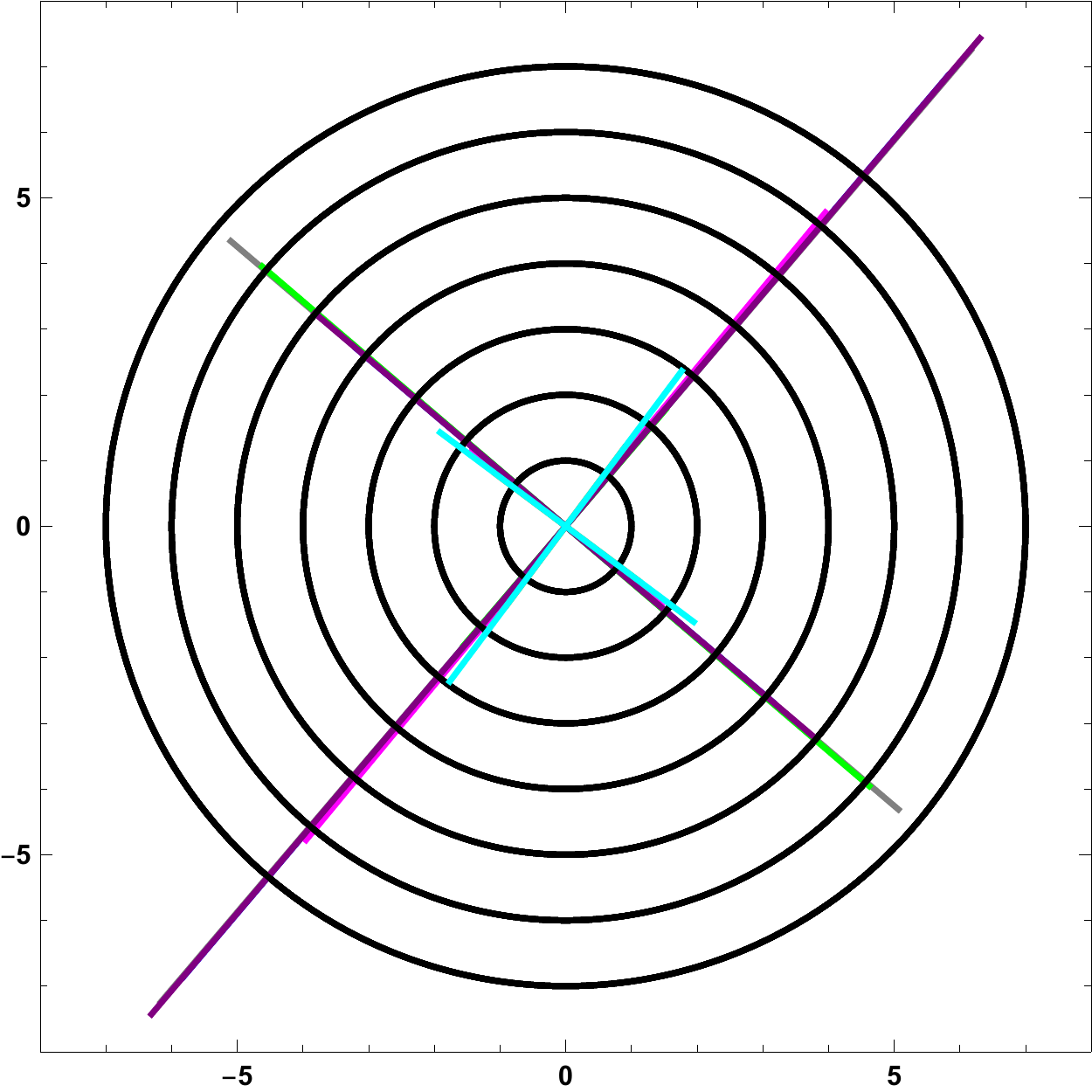}
\caption[]{The snapshots of the evolution of a set of eight eigenframes for times $t=0,.001,.002,.003,.004,.005$. The initial conditions are given in the text.}
\label{fig:BigEvolution2x2}
\end{figure}

\section{Twist States}
\label{Twist States}

It is well-known that, in addition to synchronous or phase-locked states the Kuramoto and related models admit twist or splay states, those where $\theta_j = \frac{2 \pi (j-1) n}{N}$, with the integer $n$ representing the winding number of the solution.  Mirollo \cite{mirollo1994splay} gave a very general topological existence for proof of such states for models of the form
\[
  \frac{d\theta_j}{dt} = f_j(\theta_1,\theta_2,\ldots,\theta_N)
  \]
  where the right shift operator $\sigma: \theta_j \mapsto \theta_{j+1 \mod N}$ acts equivariantly on the time evolution operator $F_t$: $F_t \circ \sigma = \sigma F_t$. In a very different vein Ferguson \cite{Ferguson2018_twists} and Delbays, Coletta and Jacquod \cite{Delbays.Coletta.Jacquod.2017} constructed twist-type solutions for the Kuramoto model on very general graphs (which typically do not satisfy Mirollo's equivariance criteria.) Medvedev and Tang studied the stability of such states on Cayley
 and random graphs \cite{medvedev2015stability}.

The properties of these depend in an interesting way on the topology of the graph. In the case of the complete graph topology (so-called all-to-all coupling) the twist states are always unstable, whereas if the graph is a single cycle the twist states are stable as long as the winding number $n$ is small enough.

In this section we construct the analogs of the twist states for the Kuramoto flow described by Equation \eqref{eq:MatrixMoto}. Interestingly the natural analog of the twist states is not a fixed point of the motion, but rather is a dynamic state on a submanifold along which all of the matrices exhibit slow (algebraic) decay to a multiple of the identity matrix. We analyze the stability of this slow decay: by a dynamic rescaling we can express these solutions as fixed points of related equations, in essence moving along an orbit of the scaling group with the solution. In this co-scaling coordinate system we analyze the stability of the fixed point and show that it is unstable.

\begin{definition}
 We define a \textbf{twist state} in $\mathbb{R}^2$ as one for  which the matrices $\{{\bf M}_1, {\bf M}_2, ..., {\bf M}_N\}$ have eigenvectors pointing in directions which are evenly spaced around the unit circle, and the eigenvalue difference (the effective coupling strength) is constant: in other words $\Delta \lambda_j:=\lambda_j^2-\lambda_j^1 $ is independent of $j$ and $\theta_j=\frac{\pi (j-1)n}{2N}$.
 \label{def:twiststate}
\end{definition}

 Again we note the factor of $4$ in the definition of the twist states for the matrix model versus the classical Kuramoto model: as before this reflects the alignment of orthogonal frames in ${\mathbb R}^2$ as opposed to angles. We see that the righthand sides of equations $\eqref{eq:2x2EigenvalueSumDynamics} \text{ and } \eqref{eq:2x2Eigen-angleDynamics}$ are equal to zero for this choice of $\Delta \lambda$  and $\boldsymbol{ \theta}$, implying that for this initial data the eigenvectors of ${\bf M}_j$ and the traces  of ${\bf M}_j$ do not evolve in time, and thus we get a single equation $\eqref{eq:2x2CouplingDynamics}$  governing the evolution of the eigenvalue difference $\lambda_j^2-\lambda_j^1$. It is clear that the right-hand side  of equation $\eqref{eq:2x2CouplingDynamics}$ is negative, so the eigenvalue difference decays in time. The decay is clearly algebraic, and given the fact that the righthand side of equation $\eqref{eq:2x2CouplingDynamics}$ is cubic in $\Delta \lambda$ it is simple to compute that the eigenvalue difference decays like $t^{-\frac12}$. Thus we have an invariant submanifold where $\theta_j = \frac{\pi (j-1) n}{2N}$ and $\Delta \lambda_j$ is independent of $j$, and along this submanifold all of the matrices decay in time to multiples of the identity matrix: ${\bf M}_j \rightarrow \frac{\tr {\bf M}_j}{2} I.$

Our goal is to understand the stability of this simple decay along a submanifold to a general perturbation. To begin with we again note that the dynamics of the Equation \eqref{eq:MatrixMoto} is invariant under the $N$ parameter family of maps ${\bf M}_j \mapsto {\bf M}_j + \alpha_j I$, where $I$ is the identity matrix. Because of this invariance we can always assume that $\tr {\bf M}_j=0.$ Our next step is to do a dynamic rescaling to turn these decaying solutions into fixed points of a modified
equation. More specifically we let $(\lambda_j^2-\lambda_j^1)=\alpha (t)$. Then, $\eqref{eq:2x2CouplingDynamics}$ becomes
\begin{align*}
\frac{d\alpha}{dt}&=-\alpha^3 \sum \limits_{k=1}^{N}\sin^2\left[2\left(\frac{\pi n(k-1)}{2N}-\frac{\pi n (j-1)}{2N}\right)\right] \\
&=-\alpha^3 \sum \limits_{k=1}^{N} \frac{1}{2}\left[1-\cos\left(4\left(\frac{\pi n(k-1)}{2N}-\frac{\pi n (j-1)}{2N}\right)\right)\right]\\
&=-\alpha^3 \sum \limits_{k=1}^{N} \frac{1}{2}\left[1-\cos\left(\frac{2 \pi n(k-1)}{2N}-\frac{2 \pi n (j-1)}{2N}\right)\right].\\
\end{align*}
An easy calculation shows that the righthand side sums to $-\frac{N \alpha^3}{2},$ leading to the separable equation  $\frac{d \alpha}{dt}=-\alpha^3 \frac{N}{2}.$
Solving for $\alpha$, we get
\begin{equation}
\alpha(t)=\frac{\alpha(0)}{\sqrt{1+N \alpha ^2(0) t}},
\label{eq:TwistDeltaLambdaRescale}
\end{equation} and thus we see that the eigenvalues decay algebraically. In order to study the stability of this invariant submanifold we re-scale (\ref{eq:2x2CouplingDynamics}) in order to obtain fixed points. To do this, we let
${\bf M}_j(t)= \alpha (t) \tilde{{\bf M}}_j(t)$
Then,
$$\frac{d{\bf M}_j}{dt}=\alpha(t)\frac{d\tilde{{\bf M}}_j}{dt}+\frac{d\alpha}{dt}\tilde{{\bf M}}_j=\sum \limits_{k=1}^{N} \alpha^3(t)[\tilde{{\bf M}_k},[\tilde{{\bf M}}_j,\tilde{{\bf M}_k}]]$$
$$\frac{d\tilde{{\bf M}}_j}{dt}=\alpha^2(t)\sum \limits_{k=1}^{N}[\tilde{{\bf M}}_k,[\tilde{{\bf M}}_j,\tilde{{\bf M}}_k]]-\frac{\frac{d\alpha}{dt}}{\alpha}\tilde{{\bf M}}_j$$
Using the fact that $\alpha$ satisfies $\frac{d \alpha}{dt}=-\alpha^3 \frac{N}{2}$, we get
$$
\frac{d\tilde{{\bf M}}_j}{dt}=\alpha^2(t)\left(\sum \limits_{k=1}^{N}[\tilde{{\bf M}}_k,[\tilde{{\bf M}}_j,\tilde{{\bf M}}_k]]+\frac{N}{2}\tilde{{\bf M}}_j \right).
$$
In order to remove the algebraic decay from this system, we rescale time via $\alpha^2(t) dt=ds$.  That is in the rescaled time
\begin{align*}
  &ds=\frac{\alpha^2(0) dt}{(1+ N\alpha^2(0)t)}& \\
  & s=\frac{\ln(1+ N \alpha^2(0) t)}{N}&
\end{align*}
a twist solution is a fixed point of
$$\frac{d\tilde{{\bf M}}_j}{ds}=\left(\sum \limits_{k=1}^{N}[\tilde{{\bf M}}_k,[\tilde{{\bf M}}_j,\tilde{{\bf M}}_k]]+\frac{N}{2}\tilde{{\bf M}}_j \right).$$
Note that this is exactly the form of the contrained gradient flow, if we constrain the Hilbert-Schmidt norm of each matrix ${\tilde {\bf M}}_i$ to remain constant.

As we did in Section \ref{subsec:2X2case}, using the spectral decomposition of $\tilde {\bf {\bf M}}_j = {\bf R}(\theta_j){\bf \Lambda}_j {\bf R}(-\theta_j)$ yields the equations of motion for the rescaled system in spectral coordinates
\begin{align}
(\dot{\lambda_j^1}+\dot{\lambda_j^2}) &=0 \label{eq:NullModes}\\
 (\dot{\lambda_j^2}-\dot{\lambda_j^1})&= -(\lambda_j^2-\lambda_j^1)\sum \limits_{k=1}^{N}  (\lambda_k^2-\lambda_k^1)^2  \sin^2(2(\theta_k-\theta_j))+\frac{N}{2} (\lambda_j^2-\lambda_j^1)
 \label{eq:2x2RescaledTwistCouplingDynamics}\\
\dot{\theta_j}&=\frac{1}{4}\sum \limits_{k=1}^{N} (\lambda_k^2-\lambda_k^1)^2 \sin(4(\theta_k-\theta_j)).
\label{eq:2x2RescaledTwistEigen-angleDynamics}
\end{align}
To analyze the stability of the twist states we compute the Jacobian.  Recall that, due to the invariance under the map ${\bf M}_i \mapsto {\bf M}_i + \alpha_i {\bf I}$, the dynamics is trivial in the $\tr({\bf M}_i)$ directions. We can see from  Equation (\ref{eq:NullModes}) that the Jacobian has at least $N$ elements in the kernel.  What is more interesting is the remaining $2N$ degrees of freedom represented by $\theta_i, \Delta \lambda_i := \lambda_i^2 - \lambda_i^1$, which evolve non-trivially. Thus we need to analyze the $2N\times 2N$ matrix that results from the Equations \eqref{eq:2x2RescaledTwistCouplingDynamics} and \eqref{eq:2x2RescaledTwistEigen-angleDynamics}.
Defining $\Delta \lambda_j := (\lambda_j^2-\lambda_j^1)$, the equations of motion become
$$
\frac{d \theta_j}{dt}=g_j(\boldsymbol{\theta}, {\Delta \lambda}) \text{ and } \frac{d \Delta \lambda_j}{dt}=f_j(\boldsymbol{\theta},{\Delta \lambda}) ,
$$
where
\[
g_j(\boldsymbol{\theta}, {\Delta \lambda})=\frac{1}{4}\sum \limits_{k=1}^{N} (\Delta \lambda_k)^2 \sin(4(\theta_k-\theta_j))
\]
 and
\[
f_j(\boldsymbol{\theta},{\Delta \lambda}) =- \Delta \lambda_j \sum \limits_{k=1}^{N} (\Delta \lambda_k)^2 \sin^2(2(\theta_k -\theta_j))+\frac{N}{2} \Delta \lambda_j.
\]
Thus we find that the Jacobian ${\bf J}$ can be written in the block form
\begin{equation}
  {\bf J} = \left[\begin{array}{ccc}
                    0 & 0 & 0 \\
  0 & \frac{\partial g}{\partial \theta} & \frac{\partial g}{\partial \Delta \lambda} \\[\medskipamount]
   0 & \frac{\partial f}{\partial \theta} & \frac{\partial f}{\partial \Delta \lambda}\end{array}\right]
   \label{eq:BlockJacobian}
\end{equation}
where the partial derivatives of each non-zero block are of the form
\[
\frac{ \partial g_j}{\partial \theta_k}= \begin{cases} -\sum \limits_{k=1}^{N} (\Delta \lambda_k)^2 \cos(4(\theta_k-\theta_j)), & \mbox{if } j=k \\ \ (\Delta \lambda_k)^2 \cos(4(\theta_k-\theta_j)), & \mbox{if } j \neq k \end{cases},\]
\[
\frac{\partial g_j}{\partial \Delta \lambda_k}= \begin{cases} 0, & \mbox{if } j=k \\ \frac{1}{2} (\Delta \lambda_k) \sin(4(\theta_k -\theta_j)) & \mbox{if } j \neq k \end{cases}, \]
\begin{align*}
\frac{\partial f_j}{\partial \theta_k}&= \begin{cases} (\Delta \lambda_j)(\Delta \lambda_k)^2 \sum \limits_{k=1}^{N}  2 \sin(4(\theta_k-\theta_j)), & \mbox{if } j=k \\-(\Delta \lambda_j)(\Delta \lambda_k)^2 2 \sin(4(\theta_k-\theta_j)), & \mbox{if } j \neq k \end{cases} \\
&= \begin{cases} 0, & \mbox{if } j=k \\-(\Delta \lambda_j)(\Delta \lambda_k)^2 2 \sin(4(\theta_k-\theta_j)), & \mbox{if } j \neq k \end{cases},
\end{align*}
and
\begin{align*}
 \frac{\partial f_j}{\partial \Delta \lambda_k} &= \begin{cases}-\sum \limits_{k=1}^{N} (\Delta \lambda_k)^2 \sin^2(2(\theta_k-\theta_j))+\frac{N}{2}, & \mbox{if } j=k \\ -2(\Delta \lambda_j)(\Delta \lambda_k)^2 \sin^2(2(\theta_k-\theta_j)), & \mbox{if } j \neq k \end{cases} \\
&= \begin{cases}0, & \mbox{if } j=k \\ -2(\Delta \lambda_j)(\Delta \lambda_k)^2 \sin^2(2(\theta_k-\theta_j)), & \mbox{if } j \neq k \end{cases}.
\end{align*}

To compute the eigenvalues of ${\bf J}$, we take advantage of the structure of the individual blocks.  Using the evenness of the cosine function, the oddness of the sine and the equivalent forms of polar angles, we have that
\begin{itemize}
    \item $\frac{\partial g}{\partial \theta}$ and $\frac{\partial f}{\partial \Delta \lambda}$ are real symmetric matrices,
    \item $\frac{\partial g}{\partial \Delta \lambda}$ and $\frac{\partial f}{\partial \theta}$ are antisymmetric,
    \item $\frac{\partial g}{\partial \Delta \lambda}$ and $\frac{\partial f}{\partial \theta}$ are algebraically related by $\frac{\partial f}{\partial \theta}= 4\frac{\partial g}{\partial \Delta \lambda}^T$, and
    \item Each block is individually a circulant matrix.
    \end{itemize}

    Here we remind the reader of a standard fact about the spectrum of circulant matrices: see, for example, the text of Terras\cite{Terras} for details.
    \begin{lemma}
      Suppose that ${\bf A}$ is an $n \times n$  circulant matrix: in other words that ${\bf A}_{ij} = c_{(i-j) \mod n}$ for some vector ${\bf c} \in {\mathbb R}^n $. Then the eigenvectors are given by ${\bf v}^{(\alpha)}$ with ${\bf v}^{(\alpha)}_j = \exp(\frac{2 \pi i (j-1)\alpha}{n})$ and the eigenvalues are given by $\lambda ^{\alpha} = \sum_{j=1}^n c_j \exp(\frac{2 \pi i (j-1)\alpha}{n})$. In other words ${\bf A}$ is diagonalized by the discrete Fourier transform, and the eigenvalues are the discrete Fourier transform of the vector ${\bf c}$.
      \end{lemma}

      The fact that the Jacobian is of block-circulant form allows us to compute the eigenvalues of the Jacobian, leading to the following stability theorem.

\begin{theorem} Any collection $\{{\bf M}_1, {\bf M}_2, ..., {\bf M}_N\}$ of $2\times2$ matrices in a twist state (Definition \ref{def:twiststate}) is an unstable orbit of Equation \eqref{eq:MatrixMoto}. More precisely, after re-scaling \eqref{eq:TwistDeltaLambdaRescale} for the algebraic decay in the eigenvalue difference $\Delta \lambda_j$, the Jacobian \eqref{eq:BlockJacobian} evaluated at a twist state has the following eigenvalues
  \begin{equation}
    \sigma(J) = \left\{ \begin{array}{c c}-N & \text{multiplicity} ~~1 \\
                          N & \text{multiplicity} ~~2\\
                        0 & \text{multiplicity} ~~3N-3 \end{array}.\right.
    \end{equation}
\end{theorem}

\begin{proof}
  Given the block-circulant nature of the matrix the theorem follows from fairly straightforward computations.
  First note that since the matrix is block circulant the individual blocks commute (all circulant matrices have a common eigenbasis) and the eigenvectors may be assumed to be a direct sum of a vector in ${\mathbb R}^2$ with the eigenvector of a circulant matrix
  \begin{equation}
    {\bf x}^{\alpha} = \left[\begin{array}{c} 0 \\ \kappa \\ \beta \end{array}\right] \oplus {\bf v}^{(\alpha)} =  \left[\begin{array}{c} {\bf 0} \\ \kappa {\bf v}^{(\alpha)} \\ \beta {\bf v}^{(\alpha)} \end{array}\right]
\label{eq:circulantEvector}
  \end{equation}
      where, as previously, ${\bf v}^{(\alpha)}_j = \exp\left(\frac{2 \pi i (j-1)\alpha}{n}\right)$.
Let $a_\alpha$ denote the eigenvalue for block $\frac{\partial g}{\partial \theta}$ associated with the eigenvector ${\bf v}^{(\alpha)}$, $b_\alpha$ for block $\frac{\partial g}{\partial \Delta \lambda}$, and $c_\alpha$ for block $\frac{\partial f}{\partial \Delta \lambda}$.  Note that the eigenvalues of the fourth block are $4\bar{b}_\alpha$, by the algebraic relationship of the off diagonal blocks.  Then, focusing solely on the non-zero block of ${\bf J}$, we have
\[
\begin{bmatrix}\frac{\partial g}{\partial \theta} & \frac{\partial g}{\partial \Delta \lambda} \\[\medskipamount]
   4\frac{\partial g}{\partial \Delta \lambda}^T & \frac{\partial f}{\partial \Delta \lambda} \end{bmatrix} \begin{bmatrix} \kappa {\bf v}^{(\alpha)} \\ \beta {\bf v}^{(\alpha)} \end{bmatrix}
=
\begin{bmatrix} \kappa \frac{\partial g}{\partial \theta}{\bf v}^{(\alpha)} +  \beta \frac{\partial g}{\partial \Delta \lambda}{\bf v}^{(\alpha)} \\
4\kappa\frac{\partial g}{\partial \Delta \lambda}^T  {\bf v}^{(\alpha)} +  \beta \frac{\partial f}{\partial \Delta \lambda} {\bf v}^{(\alpha)} \end{bmatrix} \\
=
\begin{bmatrix} \kappa a_\alpha {\bf v}^{(\alpha)} +  \beta b_\alpha {\bf v}^{(\alpha)} \\ 4 \kappa \bar{b}_\alpha  {\bf v}^{(\alpha)} +  \beta c_\alpha {\bf v}^{(\alpha)} \end{bmatrix}  \\
=
\begin{bmatrix} (\kappa a_\alpha  +  \beta b_\alpha) {\bf v}^{(\alpha)} \\ (4 \kappa \bar{b}_\alpha  +  \beta c_\alpha) {\bf v}^{(\alpha)} \end{bmatrix} \\
\]
or equivalently
\begin{equation}
\begin{bmatrix}
a_\alpha  &  b_\alpha \\
4 \bar{b}_\alpha  &   c_\alpha
\end{bmatrix}
\begin{bmatrix}
    \kappa \\
    \beta
\end{bmatrix}
=
\lambda \begin{bmatrix} \kappa \\ \beta \end{bmatrix}.
\label{eq:ReducedTwistJacobian}
\end{equation}
We see that for each $\alpha$, our larger eigenvalue problem reduces to the associated $2\times 2$ eigenvalue problem for the matrix defined by each block's individual eigenvalue.  Doing this for all $\alpha$ will yield $2N$ eigenvalues for the Jacobian matrix ${\bf J}$. A straightforward though somewhat lengthy computation -- see Appendix \ref{app:CirculantBlockLambdas} for details -- gives the following expressions for $a_{\alpha},b_{\alpha}, c_\alpha$
\begin{align*}
  & a_\alpha = \left\{ \begin{array}{rl} N/2 & \text{for } \alpha = 1, N-1 \\
                                    0 & \text{for } \alpha \neq 1, N-1 \end{array} \right. & \\
  & b_\alpha = \left\{ \begin{array}{rl} -iN/4 & \text{for }\alpha = 1 \\
                                    iN/4 & \text{for }\alpha = N-1 \\
                                    0 & \text{for } \alpha \neq 1, N-1 \end{array} \right. & \\
  & c_\alpha = \left\{ \begin{array}{rl} -N & \text{for } \alpha = 0 \\
                                    N/2 & \text{for } \alpha = 1, N-1 \\
                                     0 & \text{for } \alpha \neq 0,1, N-1 \end{array} \right. &
\end{align*}

To finish our analysis, we return to Equation \eqref{eq:ReducedTwistJacobian} and the coefficient matrix $
\begin{bmatrix}
a_{\alpha}  &  b_{\alpha} \\
4 \bar{b}_{\alpha}  &   c_{\alpha}
\end{bmatrix}
$.
For each $\alpha$ the above matrix gives two eigenvalues; the spectrum of the linearization being the union over all $\alpha \in \{0,1,\ldots,N-1\}.$ The only values of $\alpha$ that are important are $\alpha=0,1,N-1$, since the remaining $\alpha$ values give the zero matrix.
When $\alpha=0$ the coefficient matrix is $\begin{bmatrix} 0 & 0 \\ 0 & -N \end{bmatrix}$ with eigenvalues $0$ and $-N$ and associated eigenvectors $[{\bf 0}, {\bf v}^{(0)}, {\bf 0} ]^T$ and $[{\bf 0}, {\bf 0}, {\bf v}^{(0)} ]^T$ of ${\bf J}$, respectively.  Note that ${\bf v}^{(0)}$ is the vector $[1,1,1,\ldots,1]^T$. When $\alpha=1$ the coefficient matrix is $\begin{bmatrix} N/2 & -Ni/4 \\ Ni & N/2 \end{bmatrix}$ with eigenvalues  $0$ and $N$.  The associated eigenvectors of  ${\bf J}$ are $[{\bf 0}, {\bf v}^{(1)}, -2i {\bf v}^{(1)} ]^T$ and $[ {\bf 0}, (-Ni/4){\bf v}^{(1)}, (3N/2){\bf v}^{(1)} ]^T$, respectively.  Finally for $\alpha = N-1$, we get $\begin{bmatrix} N/2 & Ni/4 \\ -Ni & N/2 \end{bmatrix}$ which again has eigenvalues $0$ and $N$.  For all other $\alpha$ the coefficient matrix is the zero matrix ${\bf 0}$.  This yields two 0 eigenvalues for each $\alpha \notin \{0,1,N-1\}$.  Recall that we earlier saw that there are an additional $N$ zero eigenvalues corresponding to the coordinates $\lambda_i^1 + \lambda_i^2 = \tr({\bf M}_i).$ Thus the $3N$ eigenvalues of ${\bf J}$ are given by
\[
\{-N, N \text{ (multiplicity }2), 0 \text{ (multiplicity } 3N-3) \}.
\]
\end{proof}

Figures \ref{fig:TwistStable} -- \ref{fig:TwistPerturbedRotation} demonstrate the dynamics of a twist state and the
instability of a twist state \eqref{def:twiststate} to perturbation.  We look at three collections of matrices; the first with the exact initial conditions for a twist state, the second being an example where the eigenvalue difference $\Delta \lambda$ is not constant throughout, and the third where the rotation angles are not quite equally spaced.

For the collection of matrices in Figure \ref{fig:TwistStable}, the eigenvalue frames are all $\pi/6$ apart and the eigenvalue difference is $\Delta \lambda = 2$.  We numerically integrate the matrix Kuramoto flow with initial conditions satisfying the conditions of Definition \ref{def:twiststate} to study the dynamics.  Figure \ref{fig:TwistStable} demonstrates there is an exact solution where the eigenframe does not change but where the difference in the eigenvalues decays algebraically. The left-most plot depicts the three eigenframes and associated eigenvalues. For each matrix we plot the eigenvectors in the $x-y$ plane, scaled so that the length is the same as the magnitude of the associated eigenvalue. One of the matrices has eigenvalues $\lambda_1=5$ and $\lambda_2=3$ with associated eigenvectors $(1,0)$ and $(0,1)$. This is represented in the
plot by a line segment of length $5$ in the $(1,0)$ direction and a segment of length $3$ in the $(0,1)$ direction. The other two matrices have eigenvalues $\lambda_1=4,\lambda_2=2$ and $\lambda_1=1,\lambda_2=3$ with associated eigenvectors $v_1=[\cos\frac{\pi}{6},-\sin\frac{\pi}{6}]^t$ and $v_2=[\sin\frac{\pi}{6},\cos\frac{\pi}{6}]^t$ and $v_1=[\cos\frac{\pi}{3},-\sin\frac{\pi}{3}]^t$ and $v_2=[\sin\frac{\pi}{3},\cos\frac{\pi}{3}]^t$ respectively.
The next two frames show the solutions at times $t=1.0$ and $t=8.0$ respectively. It is apparent that the eigenvectors have not changed, but that the eigenvalue difference for each matrix has converged to zero. In the third frame the axes are of approximately equal length indicating that each matrix is close to a multiple of the identity.
This clearly shows the dynamics on the twist manifold.

Figure \ref{fig:NotTwistDeltaLambdasDifferent} demonstrates an example where the eigenvalue difference is not the same for all three matrices.  Again the rotational difference is $\pi/6$, but two of the matrices have a $\Delta \lambda =2$ and one has $\Delta \lambda =3$.  The initial condition here is similar looking, but the dynamics is very different. With these perturbed initial conditions the eigenframes rotate and quickly align, dominated by the eigenframe with the largest eigenvalues. Note that the time-scale is much shorter here than in the previous example: the three plots depict times $t=0,0.4,1.0$, as compared with $t=0,1.0,8.0$ in the previous figure. This is because the
instability is exponential, while the decay of the non-identity portions of the matrices on the twist manifold is algebraic. Figure \ref{fig:NotTwistDeltaLambdasSlightlyDifferent} shows a similar experiment, where the variation of $\Delta \lambda_i$ is much smaller: the first and third matrices have $\Delta \lambda = 2$ while the second matrix has $\Delta\lambda=2.01$. We see a similar phenomenon as in the prior figure, though on a somewhat longer time scale. The eigenframes rotate in order to align.

Figure \ref{fig:TwistPerturbedRotation} shows one final  experiment to illustrate the instability of the twist state.
In this example the eigenvalue differences are again $2$ for each matrix, as required for a twist state, but we have perturbed one of the angles slightly.  The first and third eigenframes are at angle $0$ and $\frac{\pi}{3}$ but the  second eigenframe is at angle $11\pi/60$ off the axis.  Again we see that the eigenframes do not remain fixed, but rather they rapidly rotate to coincide.

\begin{figure}[htp]
\centering
\includegraphics[width=0.32\textwidth]{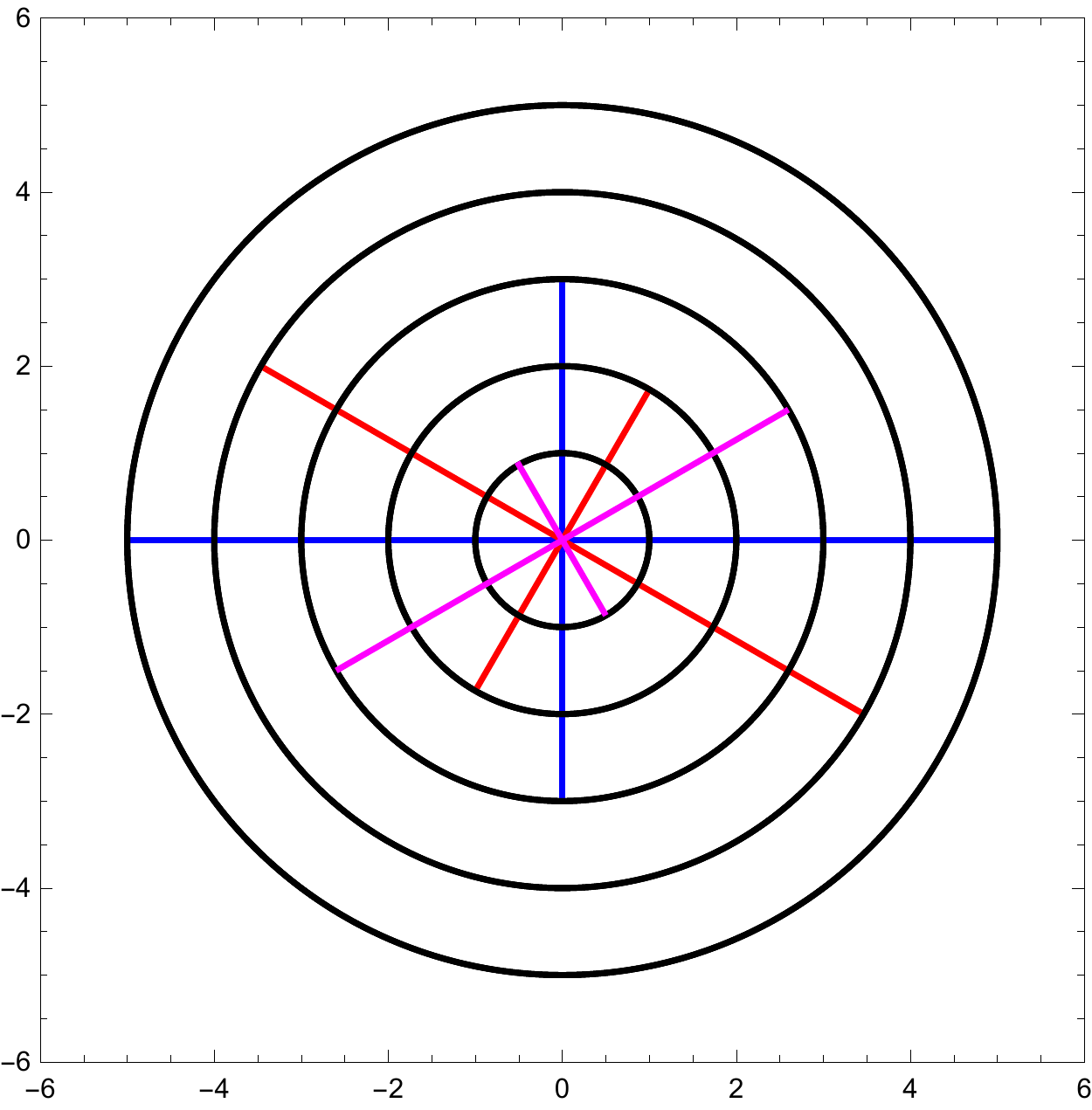}
\includegraphics[width=0.32\textwidth]{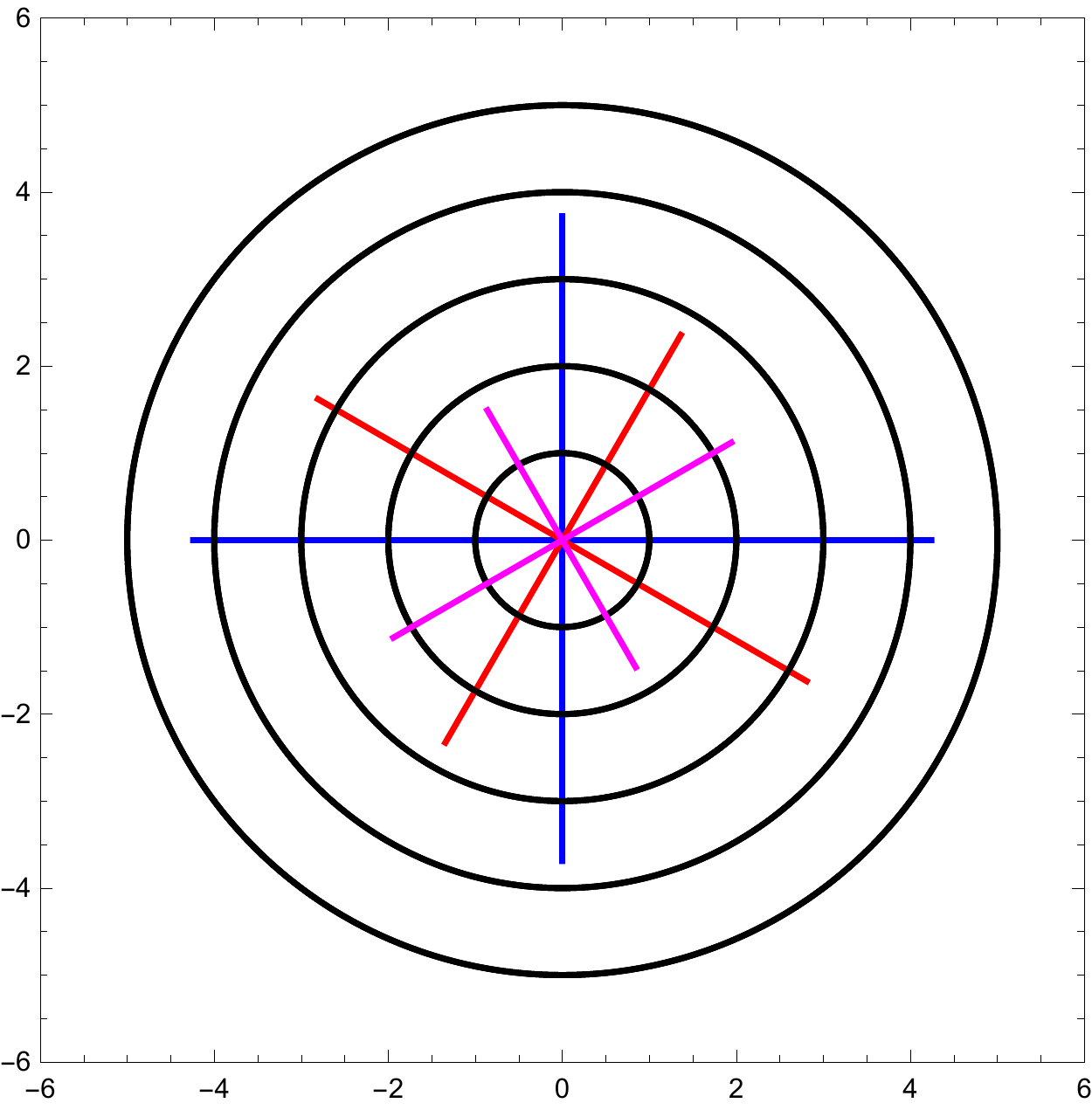}
\includegraphics[width=0.32\textwidth]{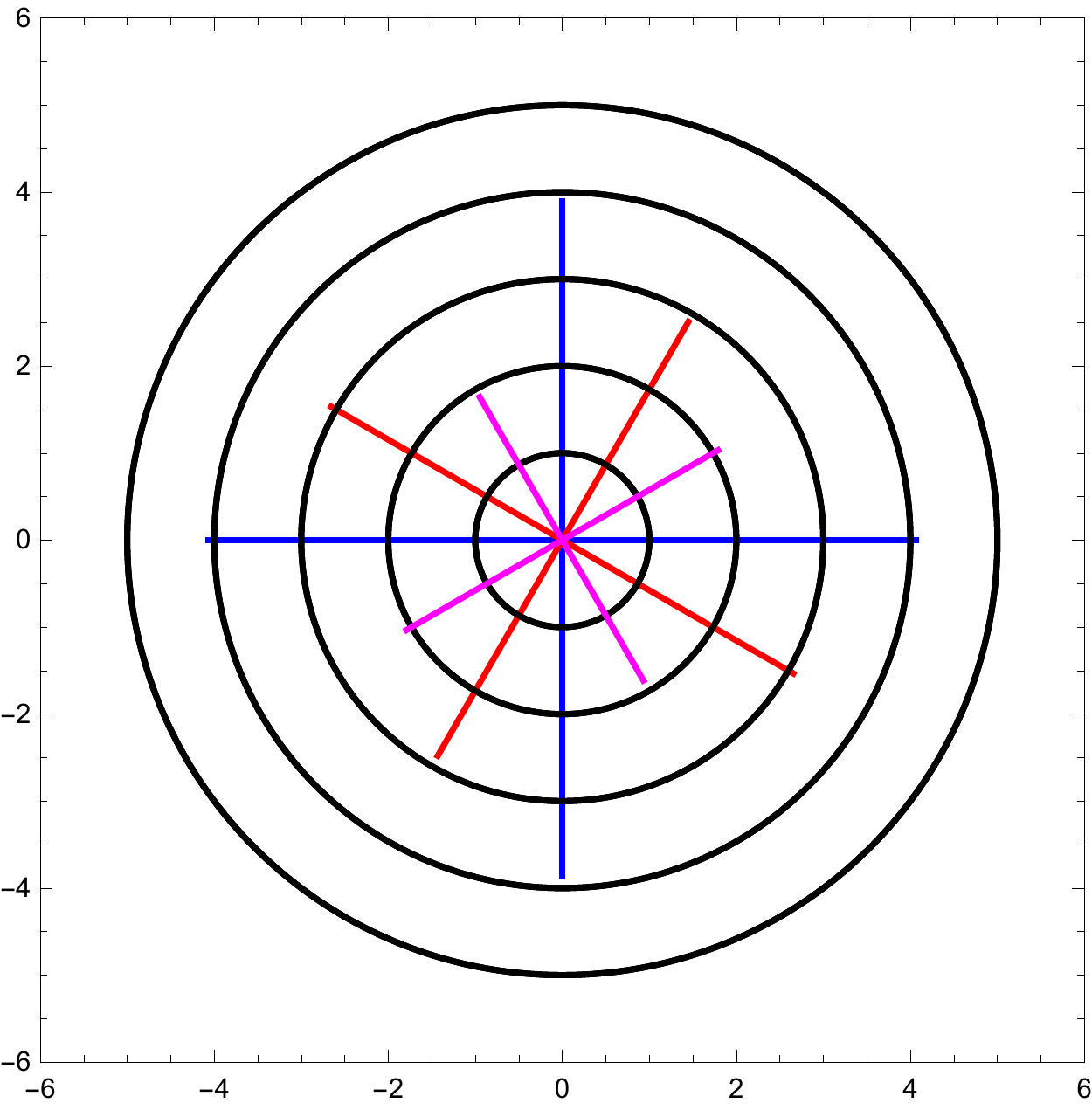}
\caption[]{This figure demonstrates the behavior of a twist state fixed point. The example begins with the initial conditions ${\bf M}_1(0) =$ $\begin{bmatrix} 5 & 0 \\ 0 & 3 \end{bmatrix}$, ${\bf M}_2(0)= {\bf R}(\frac{\pi}{6}) \begin{bmatrix} 4 & 0 \\ 0 & 2 \end{bmatrix} {\bf R}(- \frac{\pi}{6})$, and ${\bf M}_3(0)= {\bf R}(\frac{\pi}{3})\begin{bmatrix} 1 & 0 \\ 0 & 3 \end{bmatrix} {\bf R}(- \frac{\pi}{3})$.   We see that the orientation of the frames is fixed, but that the individual eigenvalues decay so that their difference is zero.  The snapshots of the evolution are at times $t=0,1,8$.}
\label{fig:TwistStable}
\end{figure}

\begin{figure}[htp]
\centering
\includegraphics[width=0.32\textwidth]{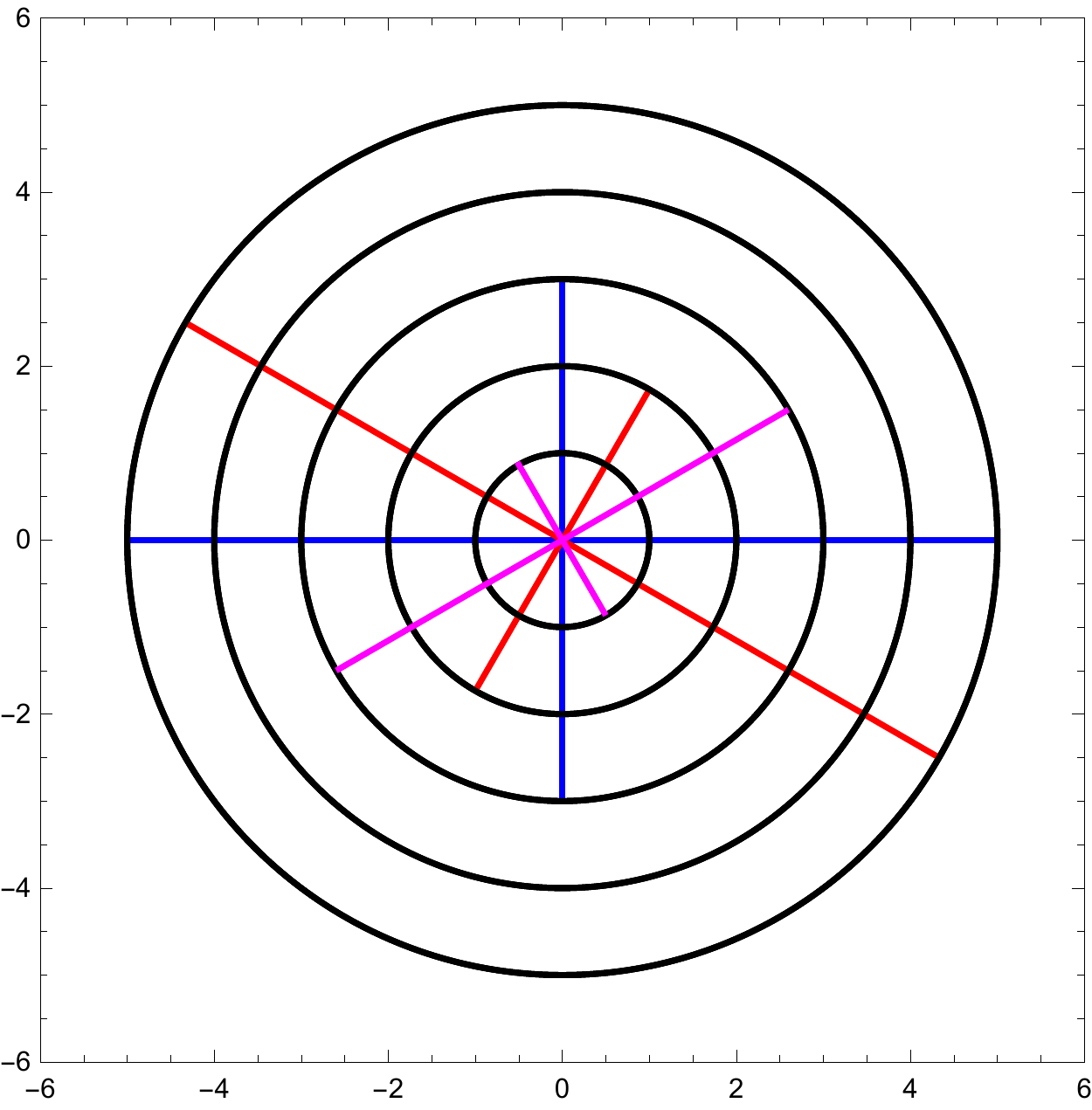}
\includegraphics[width=0.32\textwidth]{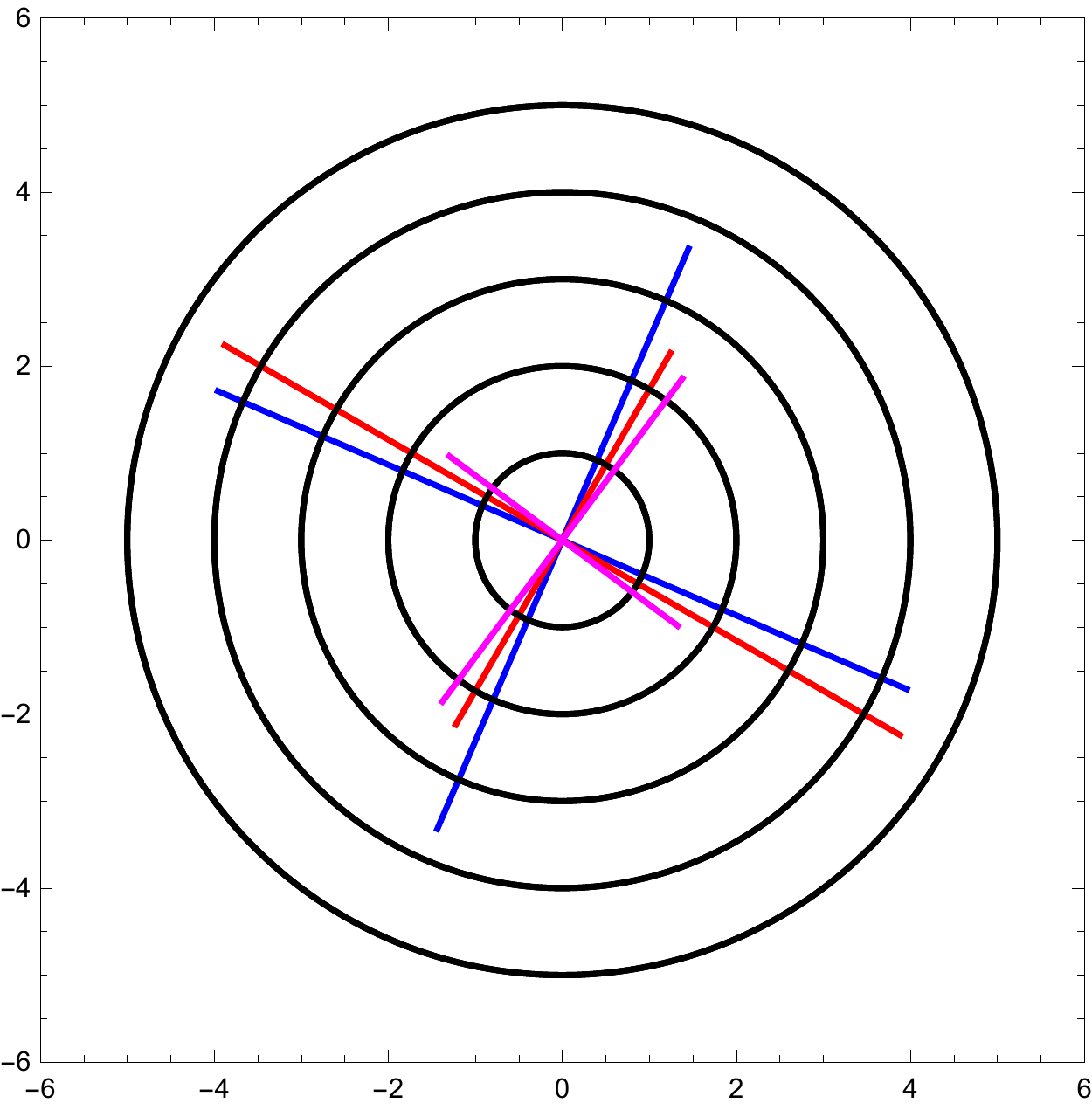}
\includegraphics[width=0.32\textwidth]{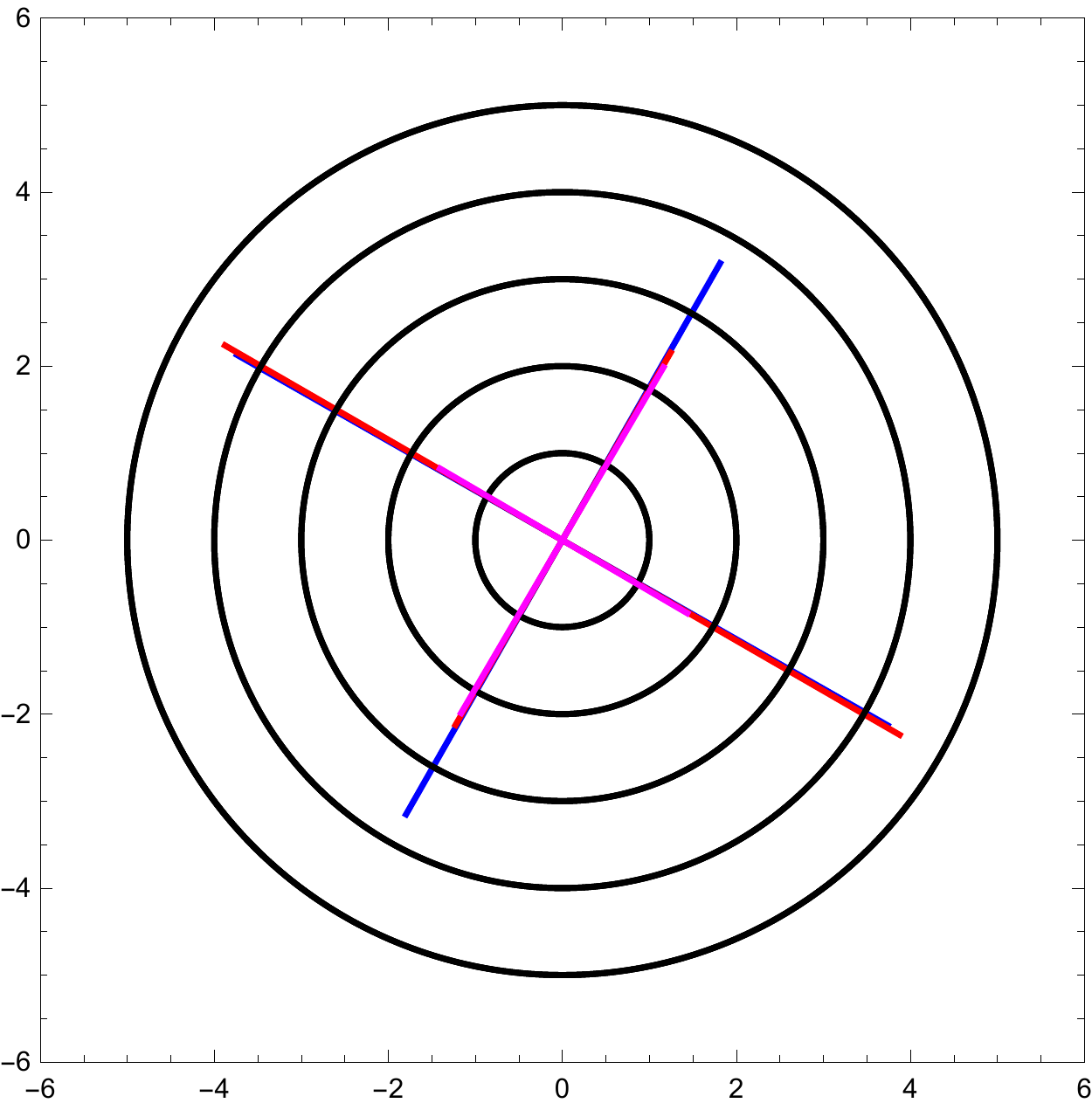}
\caption[]{This figure demonstrates the behavior when the three eigenframes are in a twist configuration (evenly space by $\pi/6$), but the initial eigenvalue difference $\Delta \lambda$ is not constant at the start.   The example begins with the initial conditions ${\bf M}_1(0) =$ $\begin{bmatrix} 5 & 0 \\ 0 & 3 \end{bmatrix}$, ${\bf M}_2(0)= {\bf R}(\frac{\pi}{6}) \begin{bmatrix} 5 & 0 \\ 0 & 2 \end{bmatrix} {\bf R}(- \frac{\pi }{6})$, and ${\bf M}_3(0)= {\bf R}(\frac{\pi}{3})\begin{bmatrix} 1 & 0 \\ 0 & 3 \end{bmatrix} {\bf R}(- \frac{\pi}{3}).$  We see that the system is unstable and evolves to aligning the eigenframes. The snapshots of the evolution are at times $t=0,0.4,1$. }
\label{fig:NotTwistDeltaLambdasDifferent}
\end{figure}

\begin{figure}[htp]
\centering
\includegraphics[width=0.32\textwidth]{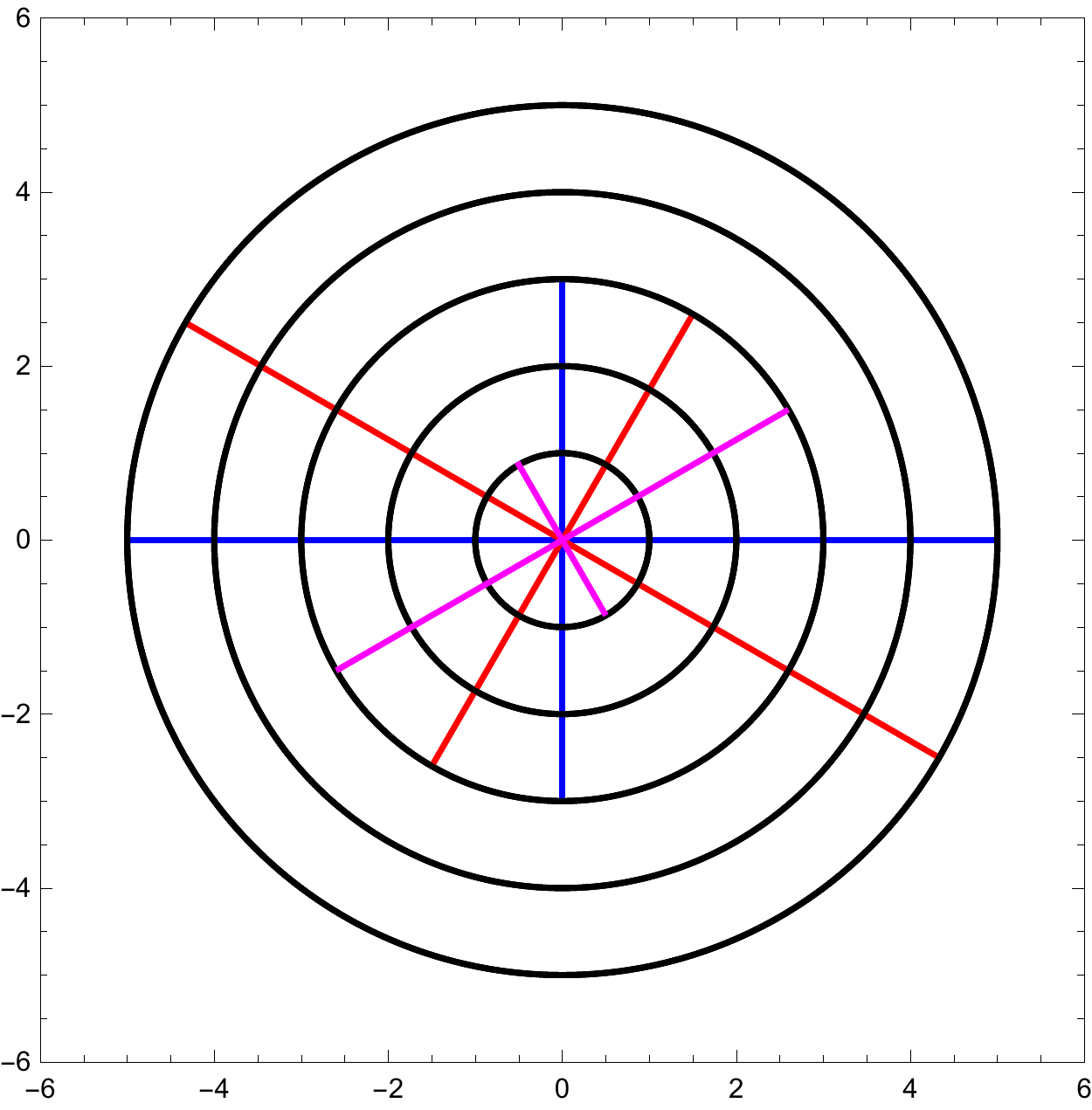}
\includegraphics[width=0.32\textwidth]{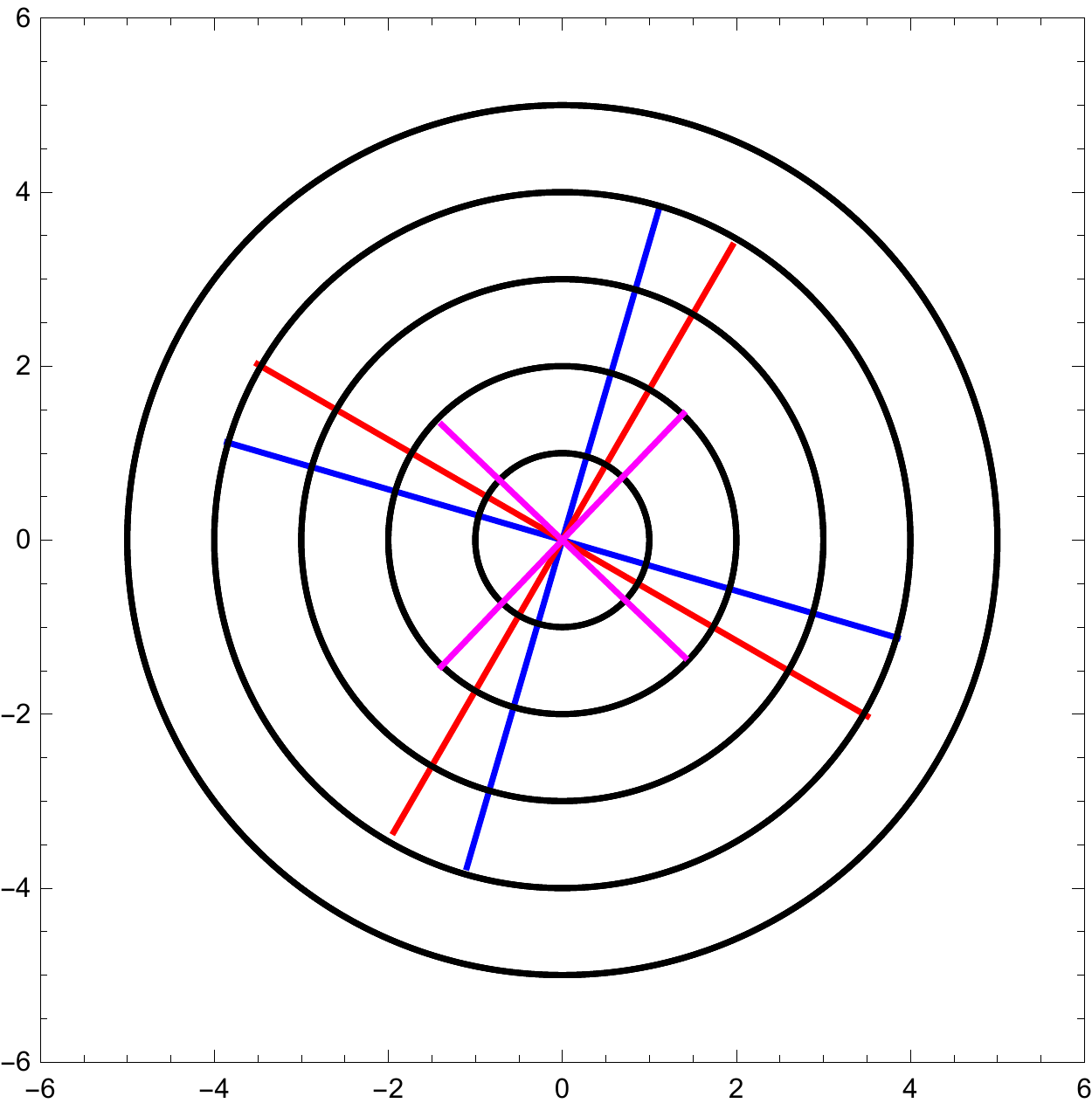}
\includegraphics[width=0.32\textwidth]{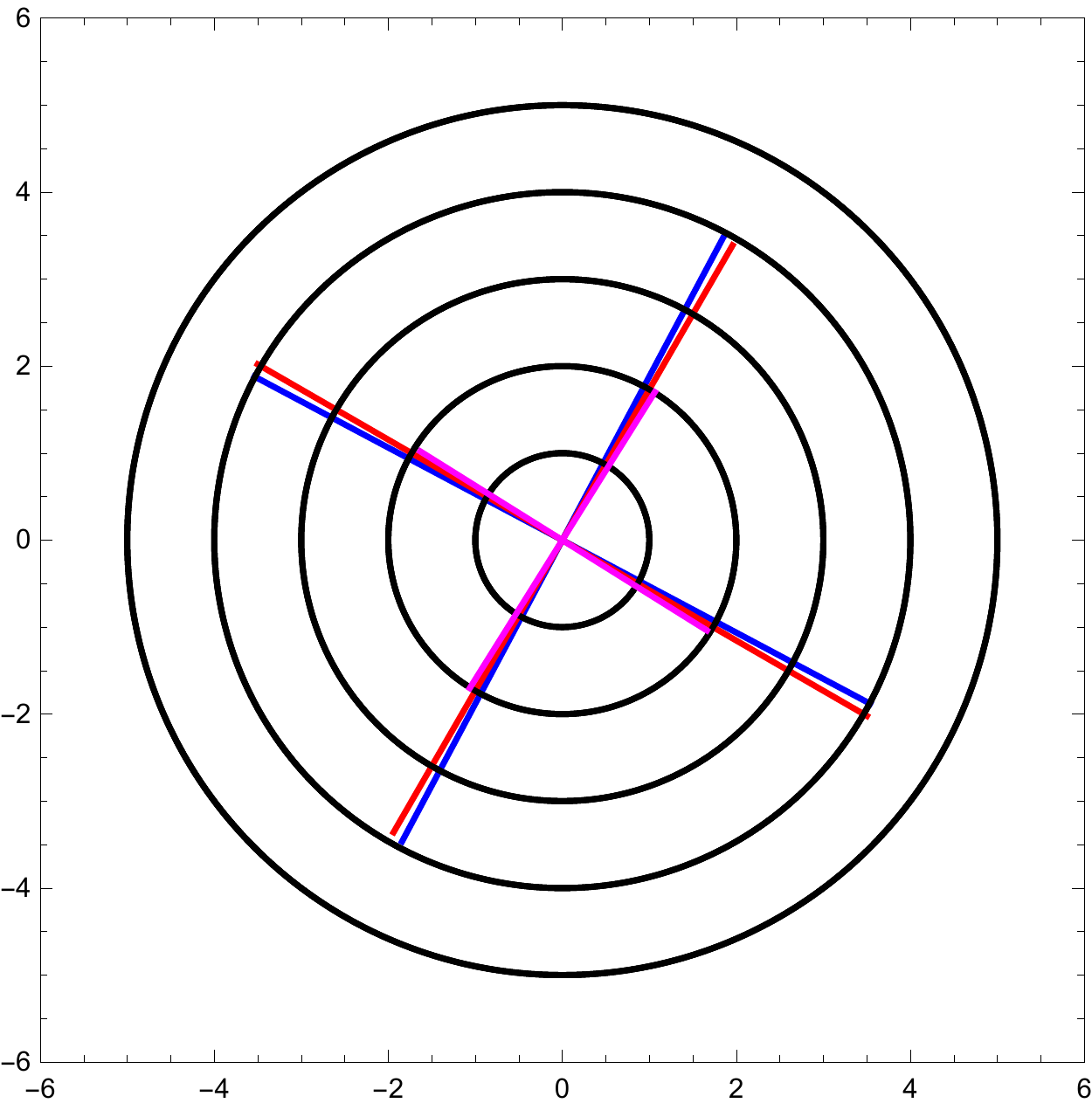}
\caption[]{This figure demonstrates the behavior when the three eigenframes are in a twist configuration (evenly space by $\pi/6$), but the initial eigenvalue difference $\Delta \lambda$ is not constant at the start.   The example begins with the initial conditions ${\bf M}_1(0) =$ $\begin{bmatrix} 5 & 0 \\ 0 & 3 \end{bmatrix}$, ${\bf M}_2(0)= {\bf R}(\frac{\pi}{6}) \begin{bmatrix} 5 & 0 \\ 0 & 2.99 \end{bmatrix} {\bf R}(- \frac{\pi }{6})$, and ${\bf M}_3(0)= {\bf R}(\frac{\pi}{3})\begin{bmatrix} 1 & 0 \\ 0 & 3 \end{bmatrix} {\bf R}(- \frac{\pi}{3}).$  We see that the system is unstable and evolves to aligning the eigenframes. The snapshots of the evolution are at times $t=0,30,80$.  }
\label{fig:NotTwistDeltaLambdasSlightlyDifferent}
\end{figure}

\begin{figure}[htp]
\centering
\includegraphics[width=0.32\textwidth]{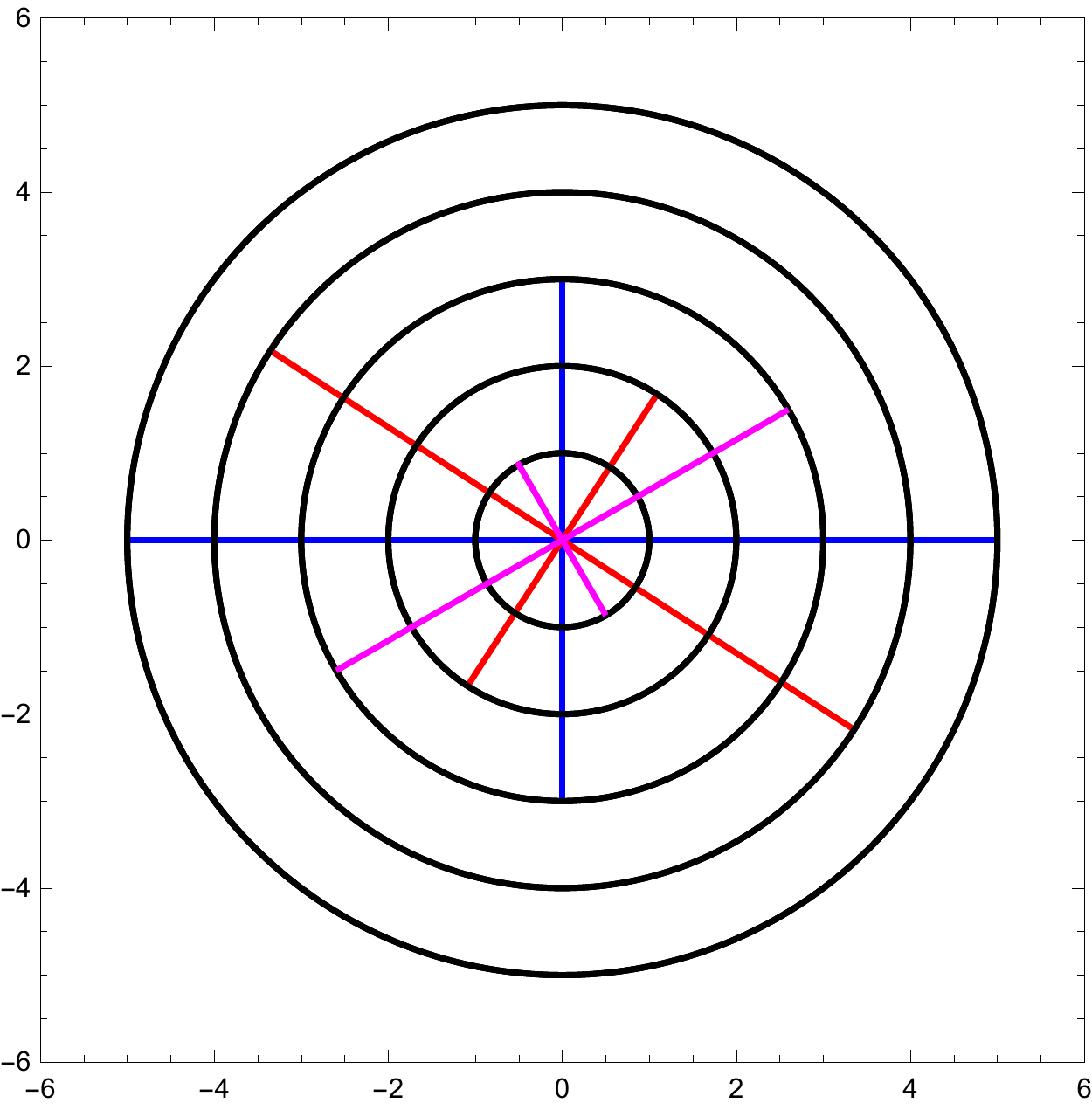}
\includegraphics[width=0.32\textwidth]{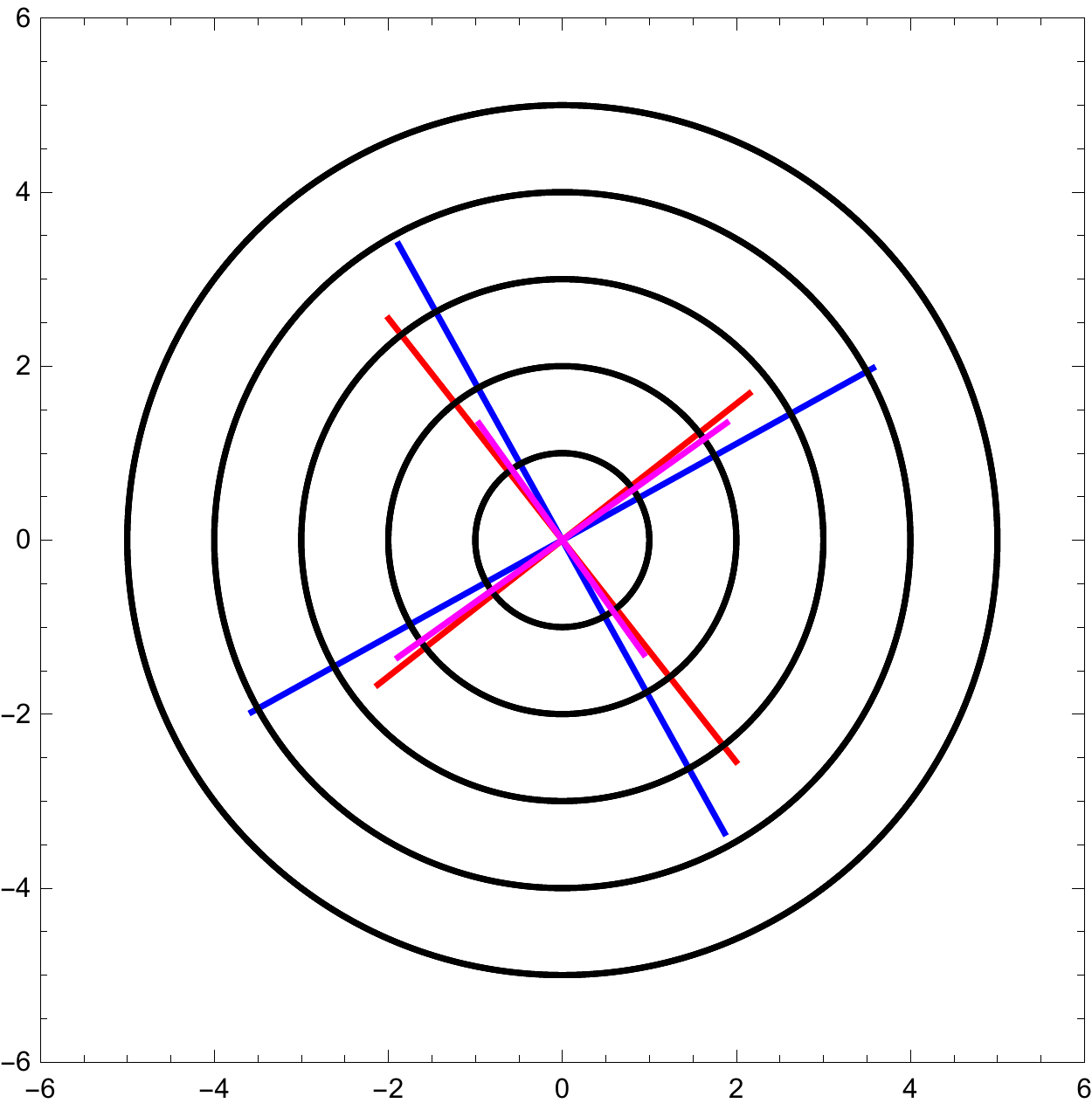}
\includegraphics[width=0.32\textwidth]{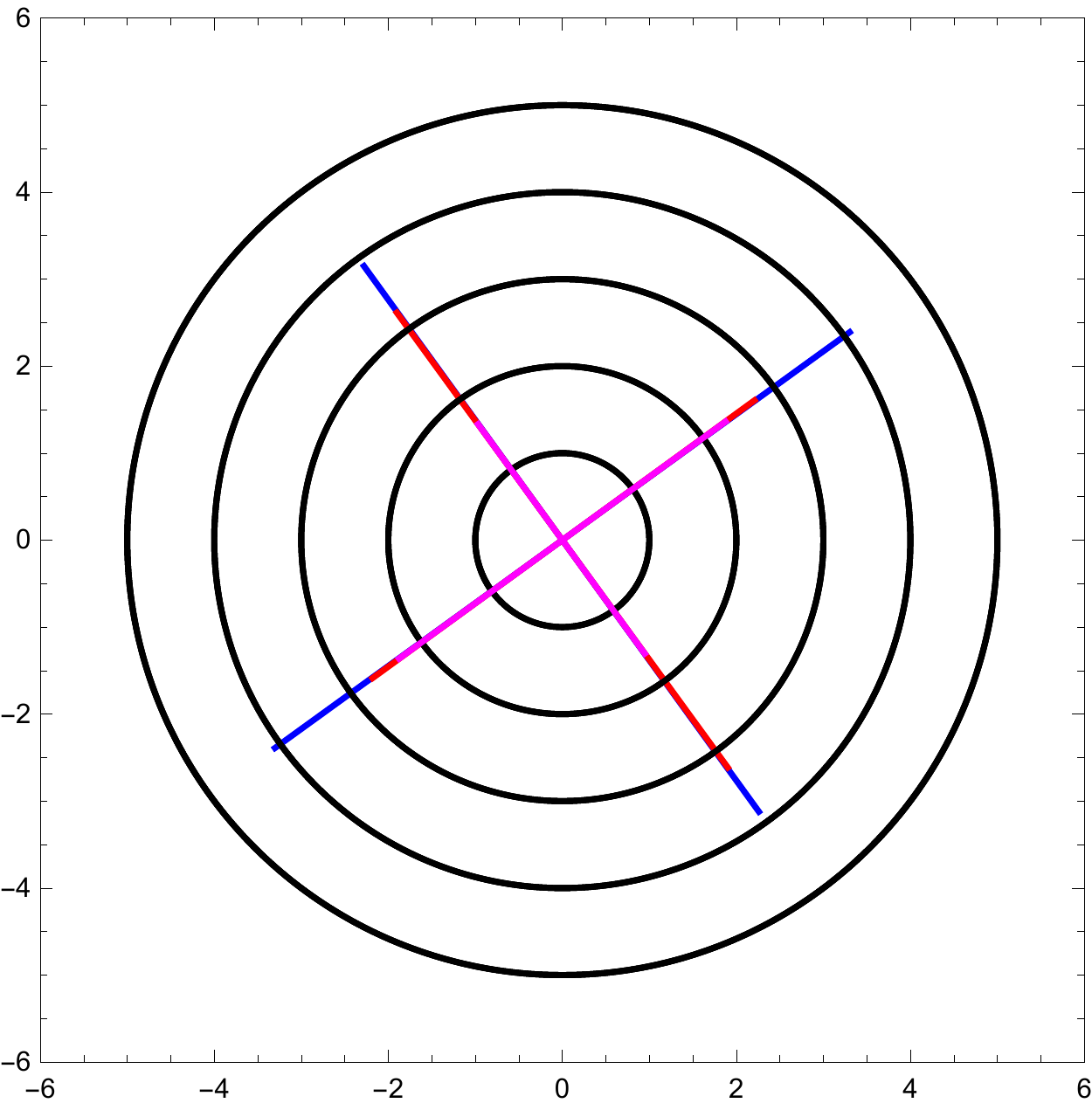}
\caption[]{This figure demonstrates the behavior when the three eigenframes are just a bit perturbed out of the $\pi/6$ alignment.  Again, we see that the initial state is unstable and the system evolves to aligning all the eigenframes.  The initial conditions are ${\bf M}_1(0) =$ $\begin{bmatrix} 5 & 0 \\ 0 & 3 \end{bmatrix}$, ${\bf M}_2(0)= {\bf R}(\frac{11 \pi}{60}) \begin{bmatrix} 4 & 0 \\ 0 & 2 \end{bmatrix} {\bf R}( \frac{-11 \pi }{60})$, and ${\bf M}_3(0)= {\bf R}(\frac{\pi}{3})\begin{bmatrix} 1 & 0 \\ 0 & 3 \end{bmatrix} {\bf R}(- \frac{\pi}{3})$ and the snapshots of evolution are at times $t=0,3,8$. }
\label{fig:TwistPerturbedRotation}
\end{figure}

\section{Stability of Commuting Matrices Fixed Points} \label{Stability of Commuting Matrices}

It is clear that any set of $N$ pairwise commuting $n\times n$ matrices $\{{\bf M}_i\}_{i=1}^N$ gives a fixed point of \eqref{eq:MatrixMoto} or (more correctly) a point on a fixed manifold. Our goal in this section is to show that this point is conditionally stable - the linearization of the flow about this point is negative semi-definite, with the dimension of the null-space equal to the dimension of the manifold of fixed points. To begin with we first count the dimension of the fixed manifold. In addition to setting our expectations with respect to the dimension of the kernel of the linearized operator it will help to motivate the construction of the basis that diagonalizes the linearization. First of all we claim that the manifold of commuting fixed points has dimension $nN + \frac{n(n-1)}{2}.$ To see this note that a family of symmetric matrices is commuting if and only if they can be simultaneously diagonalized, so each family is parameterized by $n$ eigenvalues for each of the $N$ matrices together with a common eigenbasis given by a matrix in $SO(n).$ Thus we expect that the linearized operator will have a null-space of dimension $nN + \frac{n(n-1)}{2}$. We are interested in non-zero eigenvalues of the linearization, as these determine the stability of the fixed point.
We begin by presenting a small lemma, which gives a useful representation for the
tangent space to the space of symmetric matrices. Of course the vector space of symmetric $n\times n$ matrices is isomorphic to ${\mathbb R}^{\frac{n(n+1)}{2}}$, and the tangent space to ${\mathbb R}^k$ is ${\mathbb R}^k$, but it is convenient to
give an explicit decomposition into {\em commuting} and {\em non-commuting} directions.

\begin{lemma}
  Given a real symmetric $n\times n$  matrix ${\bf M}$ with distinct eigenvalues $\lambda_i$ and orthonormal eigenvectors ${\bf v}_i$. There is an orthonormal
  decomposition of the space of symmetric $n\times n$ matrices into commuting and
  non-commuting components as follows: there is an $n$ dimensional subspace of matrices which commute with ${\bf M}$ and an $\frac{n(n-1)}{2}$ dimensional subspace of
  matrices which do not commute with ${\bf M}$ spanned respectively by $\{{\bf C}_\alpha\}_{\alpha=1}^n$ and
  $\{{\bf N}_{\alpha\beta}\}_{\alpha,\beta~\alpha>\beta}^n$
  \begin{align*}
    & {\bf C}_\alpha = {\bf v}_\alpha \otimes {\bf v}_\alpha & \\
    & {\bf N}_{\alpha\beta} = \frac1{\sqrt{2}} \left({\bf v}_\alpha \otimes {\bf v}_\beta + {\bf v}_\beta \otimes {\bf v}_\alpha\right) & \qquad i > j.
  \end{align*}
  This basis is orthonormal under the usual inner product on symmetric matrices $\langle A,B \rangle = \tr(A^TB).$
  \end{lemma}
  \begin{proof}
    It is easy to see the $ {\bf M} {\bf C}_\alpha = {\bf M} {\bf v}_\alpha \otimes {\bf v}_\alpha = \lambda_\alpha {\bf v}_\alpha \otimes {\bf v}_\alpha =   {\bf C}_\alpha {\bf M}$ so $[{\bf C}_\alpha,{\bf M}]=0$ so the matrices commute. Similarly we have that
    \begin{equation}
      [{\bf M},{\bf N}_{\alpha\beta}] = \frac{\lambda_\alpha - \lambda_\beta}{\sqrt{2}} \left({\bf v}_\alpha \otimes {\bf v}_\beta - {\bf v}_\beta \otimes {\bf v}_\alpha\right)
    \end{equation}
    which is clearly non-zero if $\lambda_\alpha$ and $\lambda_\beta$ are distinct. (In the non-generic case where ${\bf M}$ has repeated eigenvalues there are ``extra'' directions that commute with ${\bf M}$.) As the eigenvectors $\{{\bf v}_i \}$ can be chosen to be orthonormal the basis vectors ${\bf C}_\alpha, {\bf N}_{\alpha\beta}$ inherit this property.
    \end{proof}
Next we consider the linearized flow.  We assume that we have a family of $N$ commuting $n \times n$ matrices $\{{\bf M}^0_i\}_{i=1}^N$ with  $[{\bf M}_i^0,{\bf M}_j^0]=0$ for any fixed $i,j$, which forms a fixed point of the matrix Kuramoto flow. We will let $\tilde{{\bf M}}_i$ represent the perturbation of ${\bf M}_i$: in other words ${\bf M}_i = {\bf M}_i^0 + \tilde{\bf M}_i.$ With this notation it is simple to compute that $\tilde{{\bf M}}_i$ satisfies the following linearized flow
\begin{equation}
\frac{d \tilde{{\bf M}}_i}{dt}=\sum \limits_j [{\bf M}_j^0, [\tilde{{\bf M}}_i, {\bf M}_j^0]] + [{\bf M}_j^0,[{\bf M}_i^0,\tilde{{\bf M}}_j]]  \label{eq:2.15}
\end{equation}
which we can write more compactly as
\begin{equation}
  \frac{d\tilde{\bf M}}{dt} =\boldsymbol {\mathcal L}\tilde{\bf M}
  \end{equation}
 where ${\mathcal L}$ is a linear operator on the space of $N$ symmetric $n \times n$ matrices (isomorphic to ${\mathbb R}^{N\frac{n(n+1)}{2}})$.
Interestingly we can compute the spectrum of this linear operator quite explicitly -- this is the content of the next proposition.

\begin{proposition}
For a given set of commuting matrices $\{{\bf M}_i^0\}_{i=1}^N$ the spectrum of $\boldsymbol{\mathcal L}$ as an operator from ${\mathbb R}^{N\frac{n(n+1)}{2}}$ to ${\mathbb R}^{N\frac{n(n+1)}{2}}$ is given as follows. Suppose that ${\bf M}_i^0$ has eigenvalues $\lambda_i^\alpha$ for $\alpha \in \{1,2,\ldots,n\}$.   The eigenvalues of ${\mathcal L}$ are given by
\begin{align}
 0 \qquad & \text{multiplicity} ~~~nN + \frac{n(n-1)}{2} & \\
-\sum_{j=1}^N (\lambda_j^\alpha - \lambda_j^\beta)^2  \qquad & \text{multiplicity}~~~ N-1 ~~\text{for all} ~~\alpha,\beta \in \{1\ldots n\}; ~~\alpha> \beta.&
\end{align}
So generically $\boldsymbol{\mathcal L}$ has a kernel of dimension $nN + \frac{n(n-1)}{2}$ and $\frac{n(n-1)}{2}$ negative eigenvalues, each of multiplicity $N-1$. Of course for non-generic situations we could get additional degeneracies or additional elements of the kernel.

\end{proposition}
\begin{proof}
Commuting matrices share a common set of eigenvectors, a fact we use to construct a particularly convenient  orthonormal basis for the set of  $n \times n$ symmetric matrices.
Given a set $\{{\bf v}_\alpha \}_{\alpha=1}^n$ of orthonormal eigenvectors of the commuting family of matrices $\{{\bf M}_i^0\}_{i=1}^N$ we define a basis for for $n\times n$ symmetric matrices as follows.
\[
\{ {\bf v}_\alpha \otimes {\bf v}_\alpha \}_{\alpha=1}^n \cup \left\{ \frac{1}{\sqrt{2}} {\bf v}_\alpha \otimes {\bf v}_\beta + \frac{1}{\sqrt{2}} {\bf v}_\beta \otimes {\bf v}_\alpha\right\}_{\alpha>\beta}.
\]
It is straightforward to check that (assuming that $\{{\bf v}_\alpha \}_{\alpha=1}^n$ are orthogonal) this basis is orthonormal under the usual matrix inner product $\langle A,B\rangle = \tr(A^T B)$.
Using this basis, $\tilde{{\bf M}}_i$ can be written as
$$\tilde{{\bf M}}_i(t) = \sum \limits_{\alpha=1}^{k} a_\alpha^i (t) {\bf v}_\alpha \otimes {\bf v}_\alpha + \sum \limits_{\alpha>\beta} \frac{b_{\alpha_\beta}^i (t) }{\sqrt{2}}({\bf v}_\alpha \otimes {\bf v}_\beta+{\bf v}_\beta \otimes {\bf v}_\alpha),$$
and we have the following commutation rules:
\begin{align*}
 [{\bf M}_i^0, {\bf v}_\alpha \otimes {\bf v}_\alpha ] &= 0, \\
 [{\bf M}_i^0 , {\bf v}_\alpha \otimes {\bf v}_\beta + {\bf v}_\beta \otimes {\bf v}_\alpha ] &= (\lambda_i^\alpha - \lambda_i^\beta) ({\bf v}_\alpha \otimes {\bf v}_\beta - {\bf v}_\beta \otimes {\bf v}_\alpha ), \\
 [{\bf M}_i^0 , {\bf v}_\alpha \otimes {\bf v}_\beta - {\bf v}_\beta \otimes {\bf v}_\alpha ] &= (\lambda_i^\alpha - \lambda_i^\beta) ({\bf v}_\alpha \otimes {\bf v}_\beta + {\bf v}_\beta \otimes {\bf v}_\alpha ).
\end{align*}
We can see that the vectors $\{{\bf v}_\alpha \otimes {\bf v}_\alpha\}_{\alpha=1}^n$ lie in the kernel of the operator $[{\bf M}_1^0,\cdot]$: they represent the ``commuting" directions, and will be responsible for the $nN$ dimensional part of the kernel. The remaining directions $\{{\bf v}_\alpha \otimes {\bf v}_\beta + {\bf v}_\beta \otimes {\bf v}_\alpha\}_{\alpha>\beta}^n$ behave non-trivially under commutation with ${\bf M}_i^0$. Each of these will give a block, on which the linearized operator looks like a graph Laplacian. There will be one zero eigenvalue in each of these blocks, giving the additional $\binom{n}{2}$ directions in the kernel.

Given the above commutation relations we next compute the linearization using this representation of $\tilde{{\bf M}}_i$ for a fixed $i$, $j$.
\begin{align*}[{\bf M}_i^0,\tilde{{\bf M}}_j]&={\bf M}_i^0 \tilde{{\bf M}}_j-\tilde{{\bf M}}_j {\bf M}_i^0
\\
&= {\bf M}_i^0 \left(\sum \limits_{\alpha=1}^{k} a_\alpha^j {\bf v}_{\alpha} \otimes {\bf v}_{\alpha}  + \sum \limits_{\alpha> \beta} b_{\alpha_\beta}^j ({\bf v}_{\alpha} \otimes {\bf v}_{\beta}+{\bf v}_{\beta} \otimes {\bf v}_{\alpha}) \right)
\\
&\ - \left(\sum \limits_{\alpha=1}^{k} a_\alpha^j {\bf v}_{\alpha} \otimes {\bf v}_{\alpha} + \sum \limits_{\alpha> \beta} b_{\alpha_\beta}^j ({\bf v}_{\alpha} \otimes {\bf v}_{\beta}+{\bf v}_{\beta} \otimes {\bf v}_{\alpha})\right) {\bf M}_i^0 \\
&=\sum \limits_{\alpha=1}^{k} \lambda_i^\alpha a_\alpha^j {\bf v}_{\alpha} \otimes {\bf v}_{\alpha} + \sum \limits_{\alpha> \beta} b_{\alpha_\beta}^j (\lambda_i^\alpha {\bf v}_{\alpha} \otimes {\bf v}_{\beta}+ \lambda_i^\beta {\bf v}_{\beta} \otimes {\bf v}_{\alpha})
\\
&\ - \sum \limits_{\alpha=1}^{k}\lambda_i^\alpha a_\alpha^j {\bf v}_{\alpha} \otimes {\bf v}_{\alpha} - \sum \limits_{\alpha> \beta} b_{\alpha_\beta}^j (\lambda_i^\beta {\bf v}_{\alpha} \otimes {\bf v}_{\beta}+ \lambda_i^\alpha {\bf v}_{\beta} \otimes {\bf v}_{\alpha})
\\
&= \sum \limits_{\alpha > \beta} b_{\alpha_\beta}^j (\lambda_{i}^\alpha-\lambda_{i}^\beta) ({\bf v}_{\alpha} \otimes {\bf v}_{\beta}-{\bf v}_{\beta} \otimes {\bf v}_{\alpha}). \end{align*}
In a similar fashion we find that
\[
[{\bf M}_j^0,[{\bf M}_i^0, \tilde{{\bf M}_j}]] =\sum \limits_{\alpha > \beta} b_{\alpha \beta}^j (\lambda_i^\alpha - \lambda_i^\beta) (\lambda_j^\alpha-\lambda_j^\beta)({\bf v}_\alpha \otimes {\bf v}_\beta + {\bf v}_\beta \otimes {\bf v}_\alpha)
\]
and finally
\[
[{\bf M}_j^0,[\tilde{{\bf M}_i},{\bf M}_j^0]]=- \sum \limits_{\alpha > \beta} b_{\alpha \beta}^i (\lambda_j^\alpha - \lambda_j^\beta)(\lambda_j^\alpha - \lambda_j^\beta)({\bf v}_\alpha \otimes {\bf v}_\beta + {\bf v}_\beta \otimes {\bf v}_\alpha)
\]

Thus, \eqref{eq:2.15} can be written as
$$\frac{d \tilde{{\bf M}_i}}{dt}= \sum \limits_{j} \sum \limits_{\alpha > \beta} b_{\alpha \beta}^j (\lambda_i^\alpha- \lambda_i^\beta)(\lambda_j^\alpha - \lambda_j^\beta) ({\bf v}_\alpha \otimes {\bf v}_\beta + {\bf v}_\beta \otimes {\bf v}_\alpha)- \sum \limits_{j} \sum \limits_{\alpha > \beta} b_{\alpha \beta}^i (\lambda_j^\alpha - \lambda_j^\beta) (\lambda_j^\alpha - \lambda_j^\beta) ({\bf v}_\alpha \otimes {\bf v}_\beta+ {\bf v}_\beta \otimes {\bf v}_\alpha)$$
where the coefficients $a_\alpha$ and $b_{\alpha\beta}$ evolve according to
\begin{align*}
\frac{d{a}_\alpha^i}{dt}&=0\\
  \frac{d{b}_{\alpha \beta}^i}{dt}&= \sum \limits_{j} {b}_{\alpha \beta}^j (\lambda_i^\alpha-\lambda_i^\beta)(\lambda_j^\alpha - \lambda_j^\beta)-\sum \limits_{j} { b}_{\alpha \beta}^i (\lambda_j^\alpha - \lambda_j^\beta)^2
\end{align*}
Thus we have block-diagonalized the linearized equations of motion. The evolution in $a_\alpha$, the commuting direction, is trivial. We can represent the $\frac{d{\bf b}_{\alpha_\beta}}{dt}$ by
\begin{align}
\frac{d b^i_{\alpha \beta}}{dt} &=\left[(\lambda_i^\alpha-\lambda_i^\beta)(\lambda_j^\alpha-\lambda_j^\beta)-\sum \limits_{j} (\lambda_i^\alpha-\lambda_i^\beta)^2 \delta_{ij}\right] b_{\alpha \beta}^j\\
\frac{d{\bf b}_{\alpha\beta}}{dt}  &= ({\bf v}_{\alpha\beta} \otimes {\bf v}_{\alpha\beta} - \mu {\bf I}) {\bf b}_{\alpha\beta}
\end{align}
where ${\bf b}_{\alpha\beta}=[ b_{\alpha \beta}^1, b_{\alpha \beta}^2, b_{\alpha \beta}^3, \ldots b_{\alpha \beta}^N ]^T$ and the matrix ${\bf L}_{\alpha\beta}={\bf v} \otimes {\bf v} - \mu {\bf I}$ where
\[
{\bf v}=\left[(\lambda_1^\alpha - \lambda_1^\beta), (\lambda_2^\alpha - \lambda_2^\beta), (\lambda_3^\alpha - \lambda_3^\beta), ..., (\lambda_N^\alpha - \lambda_N^\beta)\right]^T
\] and $\mu=\sum \limits_{j=1}^N (\lambda_j^\alpha-\lambda_j^\beta)^2$. It is easy to see that the eigenvalues of a matrix ${\bf v} \otimes {\bf v} - \mu {\bf I}$ are $\Vert {\bf v} \Vert^2-\mu$, with multiplicity one (the corresponding eigenvector if ${\bf v}$) and $-\mu$ with multiplicity $N-1$, the eigenvectors being any orthogonal basis for the $(N-1)$ dimensional subspace orthogonal to ${\bf v}$.

Thus the eigenvalues and eigenvectors of ${\bf L}$  are:
\begin{center}
\begin{tabular}{c l}
  0 & with multiplicity 1\\
  $\displaystyle -\sum_{j=1}^N (\lambda_j^\alpha - \lambda_j^\beta)^2$ & with multiplicity $N-1$ \\
\end{tabular}
\end{center}
with corresponding eigenvectors
\begin{align*}
{\bf v}_{\alpha\beta} &= [\lambda_1^\alpha-\lambda_1^\beta,\lambda_2^\alpha-\lambda_2^\beta,\ldots,\lambda_N^\alpha-\lambda_N^\beta]^t\\
 {\bf w} &\in [\lambda_1^\alpha-\lambda_1^\beta,\lambda_2^\alpha-\lambda_2^\beta,\ldots,\lambda_N^\alpha-\lambda_N^\beta]^\perp.
\end{align*}

\end{proof}

\subsection{Numerical Simulations}

We conclude this section with some numerical experiments. It appears that for generic initial conditions in the absence of external forcing (${\bf \Omega}_i = 0$) the solutions always converge to a family of commuting matrices. We illustrate this with a couple of numerical simulations. The first numerical simulation depicts the evolution of a set of three $3 \times 3$ matrices that are initially pair-wise non-commuting. Since the eigenframes are difficult to visualize in three and higher dimensions
we instead choose to plot the Hilbert-Schmidt norms of the pairwise commutators, $\Vert [{\bf M}_i, {\bf M}_j]\Vert^2$ as functions of time. The results of one such expermient are depicted in Figure \ref{commutatornorm3x3}. We see that there is a very rapid decay of the Hilbert-Schmidt norms to zero, indicating that the matrices are asymptotically commuting. Figure \ref{commutatornorm4x4} depicts a similar experiment with three $4 \times 4$ matrices. Again we see a similar behavior -- there is a rapid decy in the Hilbert-Schmidt norms as the eigenframes move to align. We have repeated these experiments a number of times, varying both the number of matrices and the size of the matrices. In all cases we have observed very rpid convergence to a commuting set of matrices.
\begin{figure}
\includegraphics{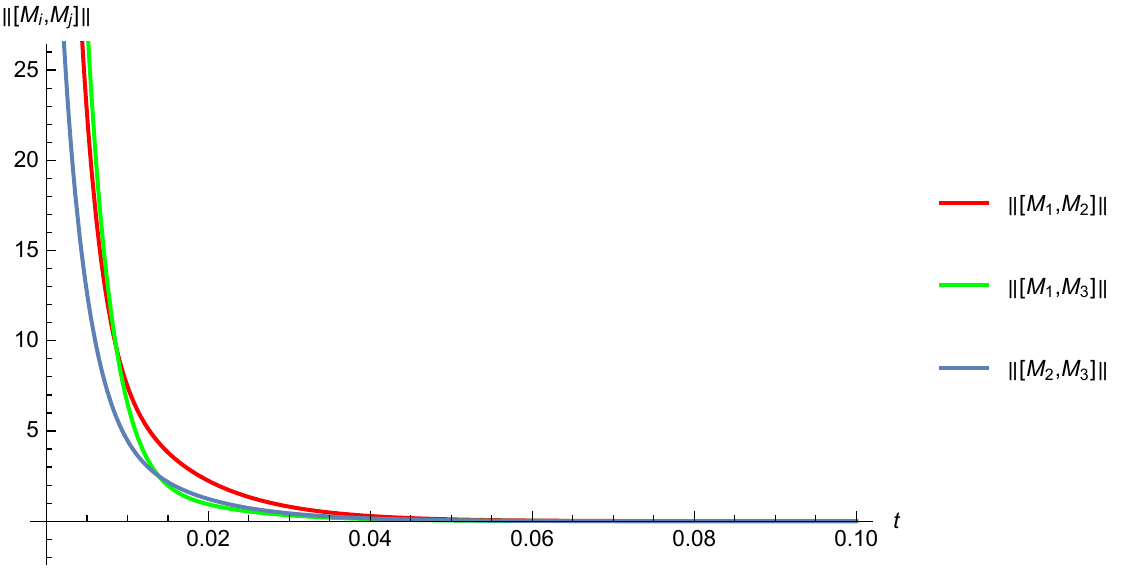}
\caption[] {The Hilbert-Schmidt norm of the pairwise commutators in a $3 \times 3$ example. The initial conditions in this example are: ${\bf M}_1(0)=$ $\begin{bmatrix} 1 & 7 & 3 \\ 7 & 4 & 5 \\ 3 & -5 & 6 \end{bmatrix}$, ${\bf M}_2(0)= \begin{bmatrix} 1 & 2 & 4 \\ 2 & 6 & 5 \\ 4 & 5 & 3 \end{bmatrix}$, and ${\bf M}_3(0)= \begin{bmatrix} 1 & 0 & 1 \\ 0 & 8 & 6 \\ 1 & 6 & 4 \end{bmatrix}$.}
\label{commutatornorm3x3}
\end{figure}

\begin{figure}
\includegraphics{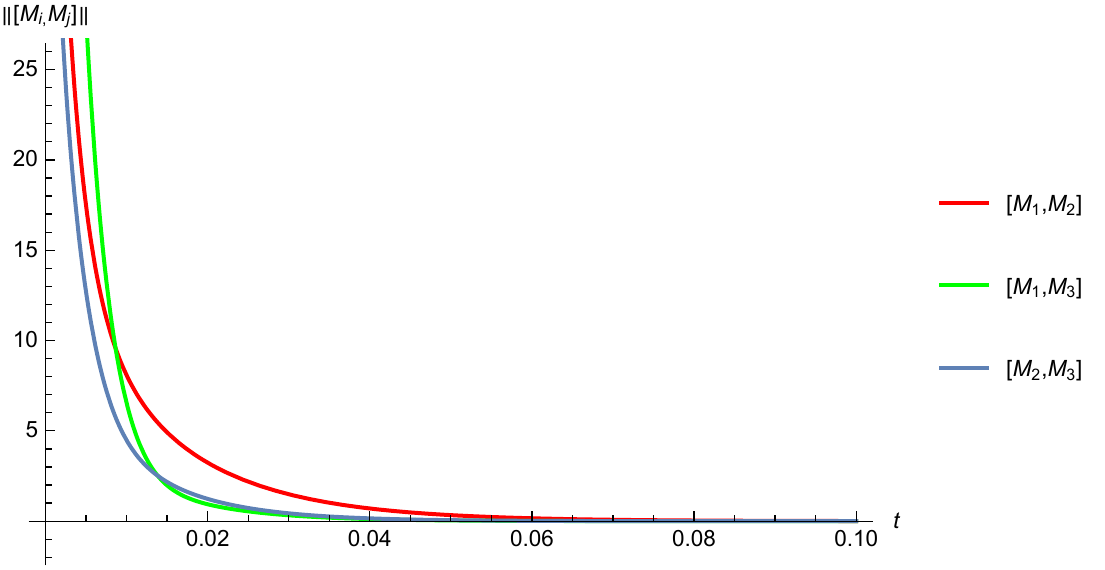}
\caption[]{The Hilbert-Schmidt norm of the pairwise commutators in a $4 \times 4$ example. The initial conditions in this example are: ${\bf M}_1(0)=\begin{bmatrix} 4 & 0 & -1 & 1 \\ 0 & 2 & 3 & 7 \\ -1 & 3 & 4 & 6 \\ 1 & 7 & 6 & 2 \end{bmatrix}$, ${\bf M}_2(0)= \begin{bmatrix} 1 & 2 & 3 & -1 \\ 2 & 4 & 6 & 2 \\ 3 & 6 & 0 & 8 \\ -1 & 2 & 8 & 1  \end{bmatrix}$, and ${\bf M}_3(0)= \begin{bmatrix} 2 & 3 & 1 & -2\\ 3 & 6 & 4 & 3 \\ 1 & 4 & -1 & 8 \\ -2 & 3 & 8 & 12 \end{bmatrix}$.}
\label{commutatornorm4x4}
\end{figure}

\section{Conclusion}

In this paper we proposed a new, purely real-valued, non-Abelian generalization of the Kuramoto model to symmetric matrix-valued variables, and proved three results for this model. We have shown that in the case of $2 \times 2$ matrices the system can be described by three equations. One of these equations is similar to the classical Kuramoto model with dynamic coupling. However, the angle difference is multiplied by 4 as a result of our consideration of the aligning of eigenvector frames as opposed to angles.  We have also defined twist states for the matrix-valued model in $\mathbb{R}^2$ to be the case in which the matrices have a constant eigenvalue difference and have eigenvectors pointing in directions which are equally spaced around the unit circle. We have proven the twist states are dynamically unstable. Additionally, we have proven that fixed points associated with families of commuting matrices of any size are conditionally stable.

Thus far we have only considered twist states in the case of $2 \times 2$ matrices. In future research, it would be interesting to investigate twist states for larger matrices. In addition, throughout this paper we have assumed all oscillators to have the same natural frequency and the absence of a constrained gradient flow.   In other words, analysis of Equation \eqref{eq:FullMatrixMoto} should yield some interesting results.

\section{Acknowledgments}

J.C.B. and S.E.S.  would like to acknowledge support under NSF grant NSF- DMS 1615418.

\noindent T.E.C. would like to acknowledge support from Caterpillar Fellowship Grant at Bradley University.

\bibliographystyle{plain}
\bibliography{MatrixMotoBib}

\appendix

\section{Derivation of the Gradient Flow}
\label{app:DerivationOfFlow}

Recall the energy function defined in Section \ref{sec:intro},
$$
E= \frac{1}{4} \sum \limits_{i,j} a_{ij} \tr[{\bf M}_i,{\bf M}_j]^2.
$$
The energy $E$ is a scalar valued function of $N$, $k\times k$  matrices ${\bf M}_i$ and $\frac{\partial E}{\partial {\bf M}_i}$ is a $k\times k$ matrix whose entries are the derivatives of $E$ with respect to the corresponding entry of ${\bf M}_i$.  Individual components of a matrix will be denoted $M_{i;\alpha\beta}$ with Greek indices for the matrix components. As in the main text the derivatives are defined
\[
\left(\frac{\partial E}{\partial {\bf M}_i}\right)_{\alpha\beta} :=  \frac{\partial E}{\partial {M}_{i;\alpha \beta}}.
\]
where $M_{i;\alpha \beta}$ is the $(\alpha, \beta)$ entry of the matrix ${\bf M}_i$.

We consider $\frac 14 \tr[{\bf M}_i,{\bf M}_j]^2$, a single term in the sum of the energy function and want to compute the derivative with respect to $M_{i;\alpha \beta}$.  Using the fact that the trace operator is linear and invariant under cyclic permutations of the matrix product, we have that
\begin{align*}
  \tr\left([{\bf M}_i,{\bf M}_j]\right)^2 &=
  2\tr\left({\bf M}_j{\bf M}_i{\bf M}_j{\bf M}_i\right)
  - 2\tr\left(({\bf M}_j)^2({\bf M}_i)^2\right) \\
                                          & =
2\sum\limits_{\gamma ,\delta , \tau , \kappa} \left(M_{j;\gamma_\delta}M_{i;\delta_\tau}M_{j;\tau_\kappa}M_{i;\kappa_\gamma}-  M_{j;\gamma_\delta}M_{j;\delta_\tau}M_{i;\tau_\kappa}M_{i;\kappa_\gamma}\right).
\end{align*}
Applying the identity that $\frac{\partial M_{j;\tau\kappa}}{\partial M_{i;\alpha\beta}} = \delta_{ij} \left(\delta_{\alpha\tau}\delta_{\beta\kappa} +\delta_{\alpha\kappa}\delta_{\beta\tau} \right)$
\begin{align*}
\frac{\partial }{\partial M_{i;\alpha \beta}}\left(\tr[{\bf M}_i,{\bf M}_j]^2 \right) & =
                                                                                                     4 \sum \limits_{\delta , \gamma} M_{j;\alpha \delta}M_{j;\delta \gamma}M_{i;\gamma \beta}\\
  &+ 4\sum \limits_{\gamma, \delta}  M_{i;\alpha \beta}M_{j;\gamma \delta }M_{j;\delta \beta} - 8 \sum \limits_{\gamma, \kappa} M_{j;\alpha \gamma}M_{i;\gamma \kappa}M_{j;\kappa \beta}
\end{align*}
Note that the resultant here is the $\alpha \beta$ entry of the matrix product
\begin{eqnarray*}
4\left( {\bf M}_j{\bf M}_j{\bf M}_i+{\bf M}_j{\bf M}_j{\bf M}_i-2{\bf M}_j{\bf M}_i{\bf M}_j\right)\\
= 4\left( {\bf M}_j[{\bf M}_j,{\bf M}_i]+[{\bf M}_i,{\bf M}_j]{\bf M}_j\right]\\
= 4[{\bf M}_j,[{\bf M}_j,{\bf M}_i].
\end{eqnarray*}
Then, for fixed $j$,
\[
\frac{\partial E}{\partial {\bf M}_i} =
 a_{ij}[{\bf M}_j,[{\bf M}_j,{\bf M}_i]].
\]
Then, summing over all the terms yields the matrix Kuramoto flow \eqref{eq:MatrixMoto}
\begin{align*}
\frac{\partial {\bf M}_i}{\partial t} &= -\frac{\partial E}{\partial {\bf M}_{i}}\\
  &=-\sum_j a_{ij}[{\bf M}_j,[{\bf M}_j,{\bf M}_i]].
\end{align*}

\section{Computation of the eigenvalues and eigenvectors of the linearization about a twist state.}
\label{app:CirculantBlockLambdas}

In order to apply Equation \eqref{eq:ReducedTwistJacobian}, for each $\alpha \in \{0,1,2,...,N-1\}$ we need to determine the associated eigenvalue for each block of ${\bf J}$, Equation \eqref{eq:BlockJacobian}.
Since the first entry of  ${\bf v}^{(\alpha)}$ is $1$ the associated eigenvalue is given by the dot product of the first row of each block with the vector ${\bf v}^{(\alpha)}$.  The following properties of the discrete Fourier transform will be useful here.
\begin{align}
\sum \limits_{j=0}^{n-1}e^{\frac{2 \pi ijk}{n}} \cos\left(\frac{2 \pi j}{n}\right) &= \begin{cases} 0, & k \neq \pm 1 \\ \frac{n}{2}, & k= \pm 1 \mod n\end{cases}\label{eq:DFTcosine}\\
\sum \limits_{j=0}^{n-1} e^{\frac{2 \pi ijk}{n}} \sin\left(\frac{2 \pi j}{n}\right) &= \begin{cases} 0, & k \neq \pm 1 \\ \frac{n}{2i}, & k=1 \mod n \\ \frac{-n}{2i}, & k=-1 \mod n \end{cases}\label{eq:DFTsine} \\
\sum \limits_{j=0}^{n-1} e^{\frac{2 \pi ijk}{n}} &= \begin{cases} n, & k=0 \mod n \\ 0, & k \neq 0 \mod n \end{cases}\label{eq:DFT}
\end{align}

We begin by looking at the $\frac{\partial g}{\partial \theta}$ block. The top row of this block is
\[
\begin{bmatrix} 1 & \cos\left(\frac{2 \pi}{N}\right) & \cos\left(\frac{4 \pi}{N}\right) & \cos\left(\frac{6 \pi}{N}\right) & ... & \cos\left(\frac{2 \pi (N-1)}{N}\right) \end{bmatrix}.
\]
To find the $k^{th}$ eigenvalue $a_k$ we take the dot product of  this row by ${\bf v}^{(k)}$ yields
\begin{align*}
a_k &=1+ e^{2 \pi i k/N} \cos\left(\frac{2 \pi}{N}\right)+ e^{4 \pi i k/N} \cos\left(\frac{4 \pi}{N}\right)+...+ e^{2 \pi i k (N-1)/N}\cos\left(\frac{2 \pi (N-1)}{N}\right)
\\
&= \sum_{j=0}^{N-1} e^{\frac{2 \pi i j k}{N}} \cos\left(\frac{2 \pi j}{N}\right)
\end{align*}
By the given identities this sum is equal to $N/2$ if $k=1$ or $k=N-1$ and is zero otherwise.

The eigenvalues for the other two blocks follow in the same way. The first row of the $\frac{\partial g}{\partial \Delta \lambda}$ block is given by
\[
\begin{bmatrix} 0 &\frac{1}{2}\sin\left(\frac{2 \pi}{N}\right) &\frac{1}{2}\sin\left(\frac{4 \pi}{N}\right) & ... & \frac{1}{2}\sin\left(\frac{2 \pi (N-1)}{N}\right) \end{bmatrix}.
\]
It is clear that since the entries of the row are given by $\sin \frac{2 \pi j}{N}$ the only non-zero Fourier coefficients will be $k=1$ and $k=N-1$ giving $b_1 = - i N/4$ and $b_{N-1} = i N/4$, with all of the remaining eigenvalues equal to zero.

Finally we consider the block $\frac{\partial f}{\partial \Delta\theta}$. The first row of this block looks like
\[
\begin{bmatrix} 0 & -1+ \cos\left(\frac{2 \pi}{N}\right) & -1 + \cos\left(\frac{4 \pi}{N}\right) & ... & -1+ \left(\cos\left(\frac{2 \pi (N-1)}{N}\right)\right)\end{bmatrix}.
\]
Since this term is the sum of a constant and a cosine term we will have three non-zero Fourier coefficients and thus three non-zero eigenvalues. Applying the above identities gives
\begin{align*}
  &c_0 = -N& \\
  & c_1 = \frac{N}{2} & \\
  & c_{N-1}= \frac{N}{2} &
\end{align*}
with all of the remaining eigenvalues zero.

\end{document}